\documentclass[11pt]{amsart}
\usepackage{verbatim, graphicx, rotating}
\usepackage{morefloats}

\usepackage{url}
\usepackage{hyperref}
\usepackage{amssymb}
\usepackage{multirow}

\input xy
\xyoption{all}
\input epsf

\newlength{\tabwidth}
\newlength{\tabheight}
\newlength{\tabrule}
\newlength{\tabwidthx}
\newlength{\tabheightx}

\def\gentabbox#1#2#3#4{\vbox to \tabheight{\setlength{\tabrule}{#3}%
  \setlength{\tabwidthx}{#1\tabwidth}\addtolength{\tabwidthx}{\tabrule}%

\setlength{\tabheightx}{#2\tabheight}\addtolength{\tabheightx}{-\tabheight}%
  \hbox to #1\tabwidth{%
    \hspace{-0.5\tabrule}\rule{\tabrule}{#2\tabheight}\hspace{-\tabrule}%
    \vbox to #2\tabheight{\hsize=\tabwidthx%
      \vspace{-0.5\tabrule}\hrule width\tabwidthx height\tabrule%
      \vspace{-0.5\tabrule}\vfil%
      \hbox to \tabwidthx{\hss#4\hss}%
        \vfil\vspace{-0.5\tabrule}%
      \hrule width\tabwidthx height\tabrule\vspace{-0.5\tabrule}}%
    \hspace{-\tabrule}\rule{\tabrule}{#2\tabheight}\hspace{-0.5\tabrule}}%
  \vspace{-\tabheightx}}}
\def\genblankbox#1#2{\vbox to \tabheight{\vfil\hbox to
#1\tabwidth{\hfil}}}

\newcommand{\field}{\mathbb}
\newcommand{\liealgebra}{\mathfrak}
\newcommand{\la}{\liealgebra}

\newcommand{\C}{{\field C}}
\newcommand{\R}{{\field R}}

\newcommand{\QU}{{\mathbb H}} 

\renewcommand{\b}{\liealgebra b}

\newcommand{\n}{{\la n}}

\newcommand{\ga}{\alpha}

\newcommand{\wt}{\widetilde}

\newtheorem{prop}{Proposition}[section]

\newtheorem{theorem}[prop]{Theorem}

\theoremstyle{definition}

\newtheorem{remark}[prop]{Remark}

\newtheorem{definition}[prop]{Definition}

\newcommand{\frb}{\mathfrak{b}}

\newcommand{\frg}{\mathfrak{g}}

\newcommand{\frk}{\mathfrak{k}}

\newcommand{\frs}{\mathfrak{s}}
\newcommand{\frt}{\mathfrak{t}}

\newcommand{\caL}{\mathcal{L}}

\newcommand{\caO}{\mathcal{O}}

\newcommand{\caT}{\mathcal{T}}

\newcommand{\caV}{\mathcal{V}}

\begin{document}
\title[$K$-orbit closures on $G/B$ as degeneracy loci]{$K$-orbit closures on $G/B$ as universal degeneracy loci for flagged vector bundles with symmetric or skew-symmetric bilinear form}

\author{Benjamin J. Wyser}
\date{\today}


\begin{abstract}
We use equivariant localization and divided difference operators to determine formulas for the torus-equivariant fundamental cohomology classes of $K$-orbit closures on the flag variety $G/B$, where $G = GL(n,\C)$, and where $K$ is one of the symmetric subgroups $O(n,\C)$ or $Sp(n,\C)$.  We realize these orbit closures as universal degeneracy loci for a vector bundle over a variety equipped with a single flag of subbundles and a nondegenerate symmetric or skew-symmetric bilinear form taking values in the trivial bundle.  We describe how our equivariant formulas can be interpreted as giving formulas for the classes of such loci in terms of the Chern classes of the various bundles.
\end{abstract}

\maketitle

Suppose that $V \rightarrow X$ is a rank $n$ vector bundle over a smooth complex variety $X$, and that $V$ is equipped with a symmetric or skew-symmetric bilinear form $\gamma$ taking values in the trivial bundle, along with a complete flag of subbundles $F_{\bullet}$.  Let $b \in S_n$ be an involution, assumed fixed point-free if $n$ is even and $\gamma$ is skew-symmetric.  Consider the degeneracy locus 
\begin{equation}\label{eqn:main-result-eqn}
	D_b = \{ x \in X \mid \text{rank}(\gamma|_{F_i(x) \times F_j(x)}) \leq r_b(i,j) \ \forall i,j\},
\end{equation}

where $r_b(i,j)$ is a non-negative integer depending on $b$, $i$, and $j$.  The main result of this paper is a recursive procedure by which one may obtain a formula for the fundamental class $[D_b] \in H^*(X)$ in the first Chern classes $c_1(F_i/F_{i-1})$ for $i=1,\hdots,n$, under certain genericity assumptions.  If $n$ is even, $\gamma$ is symmetric, and $b$ is fixed point-free, the locus $D_b$ has two irreducible components; we also describe how to obtain formulas for the fundamental classes of these components.  Such formulas involve the Chern classes of the subquotients of $F_{\bullet}$ together with an Euler class of $V$.

Although this is a nice, compact description of our results, this project was not initially motivated by a desire to find formulas for such degeneracy loci.  Rather, the motivation was to answer the following questions, in order: 
\begin{enumerate}
	\item Can torus-equivariant cohomology classes of certain orbit closures on the flag variety $G/B$ (analogous to Schubert varieties) be computed explicitly using localization techniques?
	\item Are such orbit closures universal cases of certain types of degeneracy loci, as Schubert varieties are?
	\item If so, what types of degeneracy loci are parametrized by such orbit closures?  Can a translation be made between a formula for the equivariant cohomology class of such an orbit closure and the fundamental class of such a degeneracy locus?
\end{enumerate}

In the cases considered in this paper, the answers to (1)-(2) turn out to be ``yes", and the answer to (3) turns out to be a typical locus $D_b$ as defined in (\ref{eqn:main-result-eqn}) above, with the translation between the two settings being very straightforward.  Consideration of questions (1)-(3) above was motivated by earlier work of W. Fulton (\cite{Fulton-92,Fulton-96_1,Fulton-96_2}) on Schubert loci in flag bundles, their role as universal degeneracy loci of maps of flagged vector bundles, and by connections between this work and the torus-equivariant cohomology of the flag variety, $H_T^*(G/B)$, discovered by W. Graham (\cite{Graham-97}).  We briefly describe this earlier work.  Suppose $V$ is a vector bundle over a variety $X$, and suppose that $E_{\bullet}$ and $F_{\bullet}$ are two complete flags of subbundles of $V$.  Let $w \in S_n$ be given, and consider the locus

\[ \Omega_w = \{ x \in X \ \vert \ \text{rank}(E_i(x) \cap F_j(x)) \geq r_w(i,j) \text{ for all } i,j\}, \]
where $r_w(i,j)$ is a non-negative integer depending on $w$, $i$, and $j$.  Fulton considered the problem of finding a formula for the fundamental class $[\Omega_w] \in H^*(X)$ in terms of the Chern classes of the bundles involved.  Assuming that the flags $E_{\bullet},F_{\bullet}$ are ``sufficiently generic" (in a sense that can be made precise), the problem reduces to the universal case of finding formulas for the fundamental classes of Schubert loci in the flag bundle $Fl(V)$.  Moreover, it is enough to find a formula for the smallest Schubert locus (that corresponding to a point in every fiber).  One may then deduce formulas for larger loci from this formula by applying ``divided difference operators", moving inductively up the (weak) Bruhat order.

Graham considered this problem in a more universal and Lie-theoretic setting.  Let $G$ be a reductive algebraic group over $\C$, with $T \subseteq B \subseteq G$ a maximal torus and Borel subgroup, respectively.  Denote by $E$ the total space of a universal principal $G$-bundle.  This is a contractible space with a free action of $G$ (hence also a free action of $B$, by restriction).  Let $BB$ and $BG$ denote the spaces $E/B$ and $E/G$, respectively.  Then $BB$ and $BG$ are classifying spaces for the groups $B$ and $G$.    In the setting of \cite{Graham-97}, the primary object of interest is the diagonal $\Delta \subseteq BB \times_{BG} BB$.  After a translation between $H^*(BB \times_{BG} BB)$ and the $T$-equivariant cohomology $H_T^*(G/B)$ of $G/B$, one sees that the problem of describing $[\Delta] \in H^*(BB \times_{BG} BB)$ is equivalent to that of describing the $T$-equivariant class of a point.  In the setting of $T$-equivariant cohomology, one has use of the localization theorem, which allows one to verify the correctness of a formula for the class of a point simply by checking that it restricts correctly at all of the $T$-fixed points.  The observation is that had a formula for this class not already been discovered by Fulton using other methods, it might have been determined simply by identifying how it should restrict at each fixed point and attempting to guess a class which restricts as required.

Of course, this observation is of limited use in the case of Schubert varieties, since formulas for their equivariant classes are already known, but it suggests that perhaps equivariant classes of other loci with torus actions, for which we do not already know formulas, could be computed in this way.  With this in mind, we turn now to our primary objects of interest, the closures of orbits of symmetric subgroups on $G/B$.  Let $G$ be a connected, complex, simple algebraic group of classical type.  Let $\theta$ be a (holomorphic) involution of $G$ - that is, $\theta$ is an automorphism of $G$ whose square is the identity.  Fix $T \subseteq B$, a $\theta$-stable maximal torus and Borel subgroup of $G$, respectively.  Let $K = G^{\theta}$ be the subgroup of elements of $G$ which are fixed by $\theta$.  Such a subgroup of $G$ is referred to as a \textit{symmetric subgroup}.

$K$ acts on the flag variety $G/B$ with finitely many orbits (\cite{Matsuki-79}), and the geometry of these orbits and their closures plays an important role in the theory of Harish-Chandra modules for a certain real form $G_{\R}$ of the group $G$ --- namely, one containing a maximal compact subgroup $K_{\R}$ whose complexification is $K$.  For this reason, the geometry of $K$-orbits and their closures have been studied extensively, primarily in representation-theoretic contexts.

Their role in the representation theory of real groups aside, $K$-orbit closures can be thought of as generalizations of Schubert varieties, and, in principle, any question one has about Schubert varieties may also be posed about $K$-orbit closures.  In the present paper, we apply equivariant localization as described above to discover previously unknown formulas for the $S$-equivariant fundamental classes of $K$-orbit closures on $G/B$, where $S = K \cap T$, a maximal torus of $K$ contained in $T$.  (Note that the $K$-orbit closures do not have a $T$-action, which is why we work with respect to the smaller torus.)  We do so for the symmetric pairs $(GL(n,\C),O(n,\C))$, $(SL(n,\C),SO(n,\C))$, and $(GL(2n,\C),Sp(2n,\C))$.

In each case, this is done in two steps.  First, we identify the closed orbits and their restrictions at the various $S$-fixed points.  Using this information, we produce polynomials in the generators of $H_S^*(G/B)$ which restrict at the $S$-fixed points as required.  We then conclude by the localization theorem that these polynomials represent the equivariant fundamental classes of the closed $K$-orbits.

Second, we outline how divided difference operators can be used to deduce formulas for the fundamental classes of the remaining orbit closures.  This is analogous to what is done for Schubert varieties.  Although combinatorial parametrizations of $K \backslash G/B$, as well as descriptions of its weak closure order in terms of such parametrizations, are typically more complicated than the case of Schubert varieties, in the cases treated in this paper, things are relatively straightforward.  Indeed, the orbit sets in each case can be parametrized by a subset of the Weyl group, consisting of involutions in the case $K=O(n,\C)$, and of fixed point-free involutions in the case $K=Sp(n,\C)$.

Having carried out these computations in our examples, we finally discover that the $K$-orbit closures parametrize degeneracy loci of the type described in (\ref{eqn:main-result-eqn}), by examining linear algebraic descriptions of the orbit closures as sets of flags, considering their isomorphic images in the universal space $BK \times_{BG} BB$, and by considering what sorts of additional structures on a vector bundle over a variety give rise to a classifying map into this universal space.

The paper is organized as follows:  In Section 1, we cover some preliminaries on equivariant cohomology and localization; the general means by which we hope to apply these techniques to closed $K$-orbits; and the way in which divided difference operators can then be used to determine formulas for the remaining orbit closures.  In Section 2, we carry this out explicitly in our examples.  Finally, in Section 3, we give the details of the translation between our formulas for $K$-orbit closures and Chern class formulas for degeneracy loci of the type described in (\ref{eqn:main-result-eqn}).

The results presented here are part of the author's PhD thesis, written at the University of Georgia under the direction of his research advisor, William A. Graham.  The author thanks Professor Graham wholeheartedly for his help in conceiving that project, as well as for his great generosity with his time and expertise throughout.

\section{Preliminaries}
\subsection{Notation}\label{ssec:notation}
Here we define some notations which will be used throughout the paper.

We denote by $I_n$ the $n \times n$ identity matrix, and by $J_n$ the $n \times n$ matrix with $1$'s on the antidiagonal and $0$'s elsewhere, i.e. the matrix $(e_{i,j}) = \delta_{i,n+1-j}$.  $J_{n,n}$ shall denote the block matrix which has $J_n$ in the upper-right block, $-J_n$ in the lower-left block, and $0$'s elsewhere.  That is,
\[  J_{n,n} := 
\begin{pmatrix}
0 & J_n \\
-J_n & 0 \end{pmatrix}. \]

We will use both ``one-line" notation and cycle notation for permutations.  When giving a permutation in one-line notation, the sequence of values will be listed with no delimiters, while for cycle notation, parentheses and commas will be used.  Hopefully this will remove any possibility for confusion on the part of the reader.  So, for example, the permutation $\pi \in S_4$ which sends $1$ to $2$, $2$ to $3$, $3$ to $1$, and $4$ to $4$ will be given in one-line notation as $2314$ and in cycle notation as $(1,2,3)$.

We will consider signed permutations of $\{1,\hdots,n\}$ viewed as embedded in some larger symmetric group, either $S_{2n}$ or $S_{2n+1}$, as follows:  The signed permutation $\pi$ of $\{1,\hdots,n\}$ is associated to the permutation $\sigma \in S_{2n}$ defined by
\[ \sigma(i) = 
\begin{cases}
	\pi(i) & \text{ if $\pi(i) > 0$} \\
	2n+1-|\pi(i)| & \text{ if $\pi(i) < 0$,}
\end{cases} \]
and
\[ \sigma(2n+1-i) = 2n+1-\sigma(i) \]
for $i=1,\hdots,n$.

Embedding signed permutations in $S_{2n+1}$ works the same way, with $2n$ replaced by $2n+1$ in the definitions above.  Note that this forces $\sigma(n+1) = n+1$.

When dealing with a signed permutation $w$ of $\{1,\hdots,n\}$, we will at times want to consider what we call the ``absolute value" of $w$, which we denote $|w|$.  This is defined in the obvious way, by $|w|(i) = |w(i)|$.  So for example, if $1 \overline{3} \overline{2}$ denotes the signed permutation sending $1 \mapsto 1$, $2 \mapsto -3$, and $3 \mapsto -2$, we have that $|1 \overline{3} \overline{2}| = 132$.

We will also deal often with flags, i.e. chains of subspaces of a given vector space $V$.  A flag
\[ \{0\} \subset F_1 \subset F_2 \subset \hdots \subset F_{n-1} \subset F_n = V \]
will often be denoted by $F_{\bullet}$.  When we wish to specify the components $F_i$ of a given flag $F_{\bullet}$ explicitly, we will typically use the shorthand notation
\[ F_{\bullet} = \left\langle v_1,\hdots,v_n \right\rangle, \]
which shall mean that $F_i$ is the linear span $\C \cdot \left\langle v_1,\hdots,v_i \right\rangle$ for each $i$.

We will always be dealing with characters of tori $S$ (the maximal torus of $K$) and $T$ (the maximal torus of $G$).  Characters of $S$ will be denoted by capital $Y$ variables, while characters of $T$ will be denoted by capital $X$ variables.  Equivariant cohomology classes, on the other hand, will be represented by polynomials in lower-case $x$ and $y$ variables, where the lower-case variable $x_i$ means $1 \otimes X_i$, and where the lower-case variable $y_i$ means $Y_i \otimes 1$.  (See Proposition \ref{prop:eqvt-cohom-flag-var}.)

Unless stated otherwise, $H^*(-)$ shall always mean cohomology with $\C$-coefficients.

Lastly, we note here once and for all that $K \backslash G/B$ should always be taken to mean the set of $K$-orbits on $G/B$, unless explicitly stated otherwise.  (This as opposed to $B$-orbits on $K \backslash G$, or $B \times K$-orbits on $G$.)

\subsection{Equivariant cohomology (of the flag variety), and the localization theorem}\label{ssec:eqvt_cohom}
Our primary cohomology theory is equivariant cohomology with respect to the action of a maximal torus $S$ of $K$.  The $S$-equivariant cohomology of an $S$-variety $X$ is, by definition,
\[ H_S^*(X) := H^*((ES \times X)/S). \]
Here, $ES$ denotes the total space of a universal principal $S$-bundle (a contractible space with a free $S$-action), as in the introduction.  In the next section, we will also briefly refer to $S$-equivariant homology, which is by definition the \textit{Borel-Moore} homology $H_*((ES \times X)/S)$.  (For information on Borel-Moore homology, see e.g. \cite[\S B.2]{Fulton-YoungTableaux}.)  For smooth $X$, which is all we shall be concerned with here, the two theories are identified via Poincar\'{e} duality, so we work almost exclusively with cohomology.

Note that $H_S^*(X)$ is always an algebra for the ring $\Lambda_S := H_S^*(\{\text{pt.}\})$, the $S$-equivariant cohomology of a 1-point space (equipped with trivial $S$-action).  The algebra structure is given by pullback through the constant map $X \rightarrow \{\text{pt.}\}$.

Taking $X$ to be the flag variety $G/B$, we now describe $H_S^*(X)$ explicitly.  Let $R = S(\frt^*)$, the $\C$-symmetric algebra on the dual to the Lie algebra $\frt$ of a maximal torus $T$ of $G$.  Let $R' = S(\frs^*)$, the $\C$-symmetric algebra on the dual to the Lie algebra $\frs$ of $S$.  It is a standard fact that $R \cong \Lambda_T$, and $R' \cong \Lambda_S$.  Let $n$ be the dimension of $T$, and let $r$ be the dimension of $S$.  Let $X_1,\hdots,X_n$ denote coordinates on $\frt^*$, taken as generators for the algebra $R$.  Likewise, let $Y_1,\hdots,Y_r$ denote coordinates on $\frs^*$, algebra generators for $R'$.

Note that there is a map $R \rightarrow R'$ induced by restriction of characters, whence $R'$ is a module for $R$.  Note also that $W$ acts on $R$, since it acts naturally on the characters $X_i$.  Then it makes sense to form the tensor product $R' \otimes_{R^W} R$.  As it turns out, this is the $S$-equivariant cohomology of $X$.

\begin{prop}\label{prop:eqvt-cohom-flag-var}
With notation as above, $H_S^*(X) = R' \otimes_{R^W} R$.  Thus elements of $H_S^*(X)$ are represented by polynomials in variables $x_i := 1 \otimes X_i$ and $y_i := Y_i \otimes 1$.
\end{prop}
\begin{proof}
For the case $S=T$, this is the well-known fact that $H_T^*(X) \cong R \otimes_{R^W} R$, for which a proof can be found in \cite{Brion-98_i}.  For lack of a reference in the more general case, when $S$ may be a strict subtorus of $T$, we provide a proof here, which of course applies also to the case $S=T$.

It is easy to see that $H_S^*(X)$ is free over $R'$ of rank $|W|$.  Indeed, we have a flag bundle $E \times^S (G/B) \rightarrow BS$.  This is a locally trivial fibration with fiber isomorphic to $G/B$.  On the space $E \times^S (G/B)$, for any character $\lambda \in \widehat{T}$, we have a line bundle $\caL_{\lambda}$ which restricts to the line bundle $L_{\lambda} = G \times^B \C_{\lambda}$ over the fiber $G/B$.  Express the $|W|$ Schubert classes (a basis for $H^*(G/B)$) as polynomials in the Chern classes of these line bundles.  Then those same polynomials evaluated at the Chern classes of the line bundles $\caL_{\lambda}$ give $|W|$ classes in $H^*(E \times^S G/B)$ which restrict to a basis for the cohomology of $H^*(G/B)$.  The claim now follows from the Leray-Hirsch Theorem.

Now, note that there is a map $R' \otimes_{\C} R \rightarrow H_S^*(G/B)$.  The map is the tensor product of two maps, $p: R' \rightarrow H_S^*(G/B)$ and $q: R \rightarrow H_S^*(G/B)$.  The map $p$ is pullback through the map to a point, as described above.  The map $q$ takes a character $\lambda$ to $c_1(\caL_{\lambda})$.  The map $p \otimes q$ is surjective, since the $S$-equivariant Schubert classes are hit by the map $q$ on the second factor.

Since $R$ is free over $R^W$ of rank $|W|$, $R' \otimes_{R^W} R$ is free over $R'$ of rank $|W|$, hence $H_S^*(G/B)$ and $R' \otimes_{R^W} R$ are both free $R'$-modules of the same rank.  Consider the possibility that $p \otimes q$ factors through $R' \otimes_{R^W} R$ --- that is, suppose that $x \otimes y \mapsto p(x) q(y)$ is a well-defined map $R' \otimes_{R^W} R \rightarrow H_S^*(G/B)$.   If so, then this map is clearly surjective, since $p \otimes q$ is, so it is injective as well, being a map of free $R'$-modules of the same rank.  The map is moreover a ring homomorphism, and so it is in fact an isomorphism of rings.

Thus we need only see that the map $\phi: R' \otimes_{R^W} R \rightarrow H_S^*(G/B)$ given by $\phi(\ga \otimes \beta) = p(\ga)q(\beta)$ is well-defined.  To see this, note first that the space $E \times^S (G/B)$ is isomorphic to the space $BS \times_{BG} BB$.  Indeed, the map $E \times G \rightarrow E \times_{BG} E$ given by $(e,g) \mapsto (e,eg)$ is an isomorphism, since $E \rightarrow BG$ is a principal $G$-bundle.  This map is $S \times B$-equivariant, where $S \times B$ acts on $E \times G$ by $(e,g).(s,b) = (es,s^{-1}gb)$, and on $E \times_{BG} E$ by $(e_1,e_2).(s,b) = (e_1s,e_2b)$.  Thus the isomorphism descends to quotients, and $(E \times G) / (S \times B) \cong E \times^S (G/B)$, while $(E \times_{BG} E) / (S \times B) \cong BS \times_{BG} BB$.

Now, we have a map $\wt{\phi}: H^*(BS) \otimes_{H^*(BG)} H^*(BB) \rightarrow H^*(BS \times_{BG} BB)$ given by $\alpha \otimes \beta \mapsto \pi_1^*(\alpha) \pi_2^*(\beta)$, where $\pi_1, \pi_2$ are the projections from $BS \times_{BG} BB$ onto $BS$ and $BB$, respectively.  There is no question of this map being well-defined; that it is well-defined is immediate given commutativity of the square
\[
\xymatrixcolsep{1pc} 
\xymatrixrowsep{2pc}
\xymatrix
{BS \times_{BG} BB \ar@{>}[r] \ar@{>}[d] & BB \ar@{>}[d] \\
BS \ar@{>}[r] & BG
}
\]
It is well-known that $H^*(BS) \cong R'$, $H^*(BB) \cong R$, and $H^*(BG) \cong R^W$, so clearly $H^*(BS) \otimes_{H^*(BG)} H^*(BB) \cong R' \otimes_{R^W} R$.  Thus to see that $\phi$ is well-defined, we can simply observe that it is precisely the map $\wt{\phi}$ when $H^*(BS) \otimes_{H^*(BG)} H^*(BB)$ is identified with $R' \otimes_{R^W} R$, and $H^*(BS \times_{BG} BB)$ is identified with $H_S^*(G/B) = H^*(E \times^S (G/B))$ via the isomorphism described above.

On the first factor $R'$, the map $\phi$ maps a character $\lambda$ of $S$ to $c_1((E \times^S \C_{\lambda}) \times (G/B))$.  The bundle $(E \times^S \C_{\lambda}) \times G/B$ is the line bundle associated to the principal $S$-bundle $E \times G/B \rightarrow E \times^S (G/B)$ and the $1$-dimensional representation $\C_{\lambda}$ of $S$.  On the other hand, the map $\wt{\phi}$ maps $\lambda$ to $c_1(\pi_1^*(\caL_{\lambda}))$.  The bundle $\pi_1^*\caL_{\lambda} = (E \times^S \C_{\lambda}) \times_{BG} BB$ is the line bundle associated to the principal $S$-bundle $E \times_{BG} BB \rightarrow BS \times_{BG} BB$ and the same $1$-dimensional representation $\C_{\lambda}$ of $S$.  Since these two line bundles are associated to principal $S$-bundles which correspond via our isomorphism, and to the same representation of $S$, they are in fact the same line bundle when the two spaces are identified.  Thus $\phi$ and $\wt{\phi}$ agree on the $R'$ factor.

The story on the second factor is much the same.  The map $\phi$ maps a character $\lambda$ of $T$ to $c_1(E \times^S (G \times^B \C_{\lambda}))$, the first Chern class of the line bundle associated to the principal $B$-bundle $E \times^S G \rightarrow E \times^S (G/B)$ and the $1$-dimensional representation $\C_{\lambda}$ of $B$ (where, as usual, the $T$-action on $\C_{\lambda}$ is extended to $B$ by letting the unipotent radical act trivially).  The map $\wt{\phi}$ maps $\lambda$ to $c_1(\pi_2^*\caL_{\lambda})$, with $\pi_2^*\caL_{\lambda} = BS \times_{BG} (E \times^B \C_{\lambda})$ the line bundle associated to the principal $B$-bundle $BS \times_{BG} E \rightarrow BS \times_{BG} BB$ and the same representation of $B$.  Since these principal bundles correspond via our identification, and since the line bundles are associated to these principal bundles and the same representations of $B$, they are the same line bundle.  Thus $\phi$ and $\wt{\phi}$ agree on the $R$ factor as well.
\end{proof}

As mentioned, the $S$-equivariant cohomology of any $S$-variety $X$ is an algebra for $\Lambda_S$, the $S$-equivariant cohomology of a point.  We have the following standard localization theorem for actions of tori, one reference for which is \cite{Brion-98_i}:

\begin{theorem}
Let $X$ be an $S$-variety, and let $i: X^S \hookrightarrow X$ be the inclusion of the $S$-fixed locus of $X$.  The pullback map of $\Lambda_S$-modules 
\[ i^*: H_S^*(X) \rightarrow H_S^*(X^S) \]
is an isomorphism after a localization which inverts finitely many characters of $S$.  In particular, if $H_S^*(X)$ is free over $\Lambda_S$, then $i^*$ is injective. 
\end{theorem}

The last statement is what is relevant for us, since when $X$ is the flag variety, $H_S^*(X) = R' \otimes_{R^W} R$ is free over $R'$.  Thus in the case of the flag variety, the localization theorem tells us that any equivariant class is entirely determined by its image under $i^*$.  As noted in the next section (cf. Proposition \ref{prop:s-fixed-points}), the locus of $S$-fixed points is finite, and indexed by the Weyl group $W$, even in the event that $S$ is a proper subtorus of the maximal torus $T$ of $G$.  Thus in our setup, 
\[ H_S^*(X^S) \cong \bigoplus_{w \in W} \Lambda_S, \]
so that in fact a class in $H_S^*(X)$ is determined by its image under $i_w^*$ for each $w \in W$, where here $i_w$ denotes the inclusion of the $S$-fixed point $wB$.    Given a class $\beta \in H_S^*(X)$ and an $S$-fixed point $wB$, we will typically denote the restriction $i_w^*(\beta)$ at $wB$ by $\beta|_{wB}$, or simply by $\beta|_w$ if no confusion seems likely to arise.

Suppose that $Y$ is a closed $K$-orbit.  We denote by $[Y] \in H_S^*(X)$ its $S$-equivariant fundamental class.  For the sake of clarity, we explain this abuse of notation.  To be precise, by $[Y]$ we mean the Poincar\'{e} dual to the direct image of the fundamental (equivariant) homology class of $Y$ in $H_*^S(X)$.  This is the unique equivariant cohomology class $\alpha \in H_S^*(X)$ having the property that $\alpha \cap [X] = [Y]$.

We describe in the next section how to compute $[Y]|_{wB}$ for $w \in W$.  Since $[Y]$ is completely determined by these restrictions, the idea is to compute them and then try to ``guess" a formula for $[Y]$ based on them.  For us, a ``formula for $[Y]$" is a polynomial in the variables $x_i$ and $y_i$ (defined in the statement of Proposition \ref{prop:eqvt-cohom-flag-var}) which represents $[Y]$.  Note that such a formula amounts to a particular choice of lift of $[Y]$ from $R' \otimes_{R^W} R$ to $R' \otimes_{\C} R$.

To be able to tell whether a given guess at a formula for $[Y]$ is correct, we must understand how the restriction maps $i_w^*$ work.  That is the content of the next proposition.

\begin{prop}\label{prop:restriction-maps}
Suppose that $\beta \in H_S^*(X)$ is represented by the polynomial $f = f(x,y)$ in the $x_i$ and $y_i$.  Then $\beta|_{wB} \in \Lambda_S$ is the polynomial $f(\rho(wX),Y)$.  Here, $\rho$ denotes the restriction $\frt^* \rightarrow \frs^*$.
\end{prop}
\begin{proof}
It suffices to check that 
\[ y_i|_{wB} = Y_i, \]
and that
\[ x_i|_{wB} = \rho(wX_i). \]

For the first, recall that the class $y_i \in H_S^*(X)$ is $\pi^*(Y_i)$, where $\pi: X \rightarrow \{\text{pt.}\}$ is the map to a point, and $Y_i \in \frs^*$ is a coordinate on $\frs$.  Letting $i_w$ denote the inclusion of the fixed point $wB$ into $X$, we have that $\pi \circ i_w = id$, so that $i_w^* \circ \pi^*$ is the identity on $H_S^*(\{wB\})$.  Thus $i_w^*(y_i) = i_w^*(\pi^*(Y_i)) = Y_i$, which is what is being claimed.

For the second, recall that $x_i$ is the $S$-equivariant Chern class $c_1^S(L_{X_i}) = c_1(E \times^S L_{X_i})$, with $X_i \in \frt^*$.  Thus
\[ i_w^*(x_i) = i_w^*(c_1(E \times^S L_{X_i})) = c_1(i_w^*(E \times^S L_{X_i})). \]
The bundle $i_w^*(E \times^S L_{X_i})$ over $BS$ is pulled back from the bundle $i_w^*(E \times^T L_{X_i})$ over $BT$ through the natural map $BS \rightarrow BT$.  The bundle $i_w^*(E \times^T L_{X_i})$ corresponds to a $T$-equivariant bundle over $\{wB\}$ (i.e. a representation of $T$) having weight $wX_i$, as one easily checks.  Thus the bundle $i_w^*(E \times^S L_{X_i})$ corresponds to an $S$-equivariant bundle over $\{wB\}$ having $S$-weight $\rho(wX_i)$, since the pullback $\Lambda_T \rightarrow \Lambda_S$ through the map $BS \rightarrow BT$ is determined by restriction of characters.
\end{proof}

\subsection{Closed Orbits}\label{ssec:closed_orbits}
Let $G, B, T, K, S, W$ be as in the introduction.  Let $\Phi = \Phi(G,T)$ denote the roots of $G$.  Let $\Phi^+$ denote the positive system of $\Phi$ such that the roots of $B$ are \textit{negative}, and denote $\Phi^- = -\Phi^+ = \Phi(B,T)$.  Let $X = G/B$ be the flag variety.

In our computations of equivariant classes, the closed orbits play a key role.  These are the orbits for whose classes we give formulas explicitly.  We use equivariant localization as described in the previous section to verify the correctness of these formulas.  Taking such formulas as a starting point, formulas for classes of remaining orbit closures can then be computed using divided difference operators, as explained in the next section.

In this subsection, we give the general facts regarding the closed orbits which we use to compute their equivariant classes.  By equivariant localization, to determine a formula for the $S$-equivariant class of a closed orbit, it suffices, at least in principle, to compute the restriction of this class at each $S$-fixed point.  We start by identifying the $S$-fixed points.  We know that the $T$-fixed points are finite, and indexed by $W$.  The question is whether $X^S$ can be larger than this, in the event that $S$ is a proper subtorus of $T$.  In fact, it cannot.  We refer to \cite{Brion-99} for the following result:

\begin{prop}[\cite{Brion-99}]\label{prop:s-fixed-points}
If $K = G^{\theta}$ is a symmetric subgroup of $G$, $T$ is a $\theta$-stable maximal torus of $G$, and $S$ is a maximal torus of $K$ contained in $T$, then $(G/B)^S = (G/B)^T$.
\end{prop}

With the $S$-fixed locus described, we now outline how the restriction of the class of a closed orbit to an $S$-fixed point can be computed explicitly.  The key fact that we use is the self-intersection formula.  To show that the self-intersection formula even applies, we first need the following easy result:

\begin{prop}\label{prop:closed-orbits-smooth}
	Suppose that $K$ is a connected symmetric subgroup of $G$.  Then each closed $K$-orbit is isomorphic to the flag variety for the group $K$.  In particular, any closed $K$-orbit is smooth.
\end{prop}  
\begin{proof}
	Suppose that $K \cdot gB$ is a closed orbit.  Then $K \cdot gB \cong K / \text{Stab}_K(gB)$, and clearly, $\text{Stab}_K(gB) = g^{-1}Bg \cap K$.  Because $K \cdot gB$ is a closed subvariety of $G/B$ and because $G/B$ is complete, $K \cdot gB$ is complete as well.  Thus $g^{-1}Bg \cap K$ is a parabolic subgroup of $K$ (\cite[\S 21.3]{Humphreys-75}).  Since it contained in the Borel subgroup $g^{-1}Bg$ of $G$, it is solvable, and so it is in fact a Borel subgroup of $K$.  Thus $K \cdot gB$ is isomorphic to a quotient of $K$ by a Borel.
\end{proof}

Let $Y$ be a closed $K$-orbit, with $Y \stackrel{i}{\hookrightarrow} X$ the inclusion.  Recall that what we are trying to compute is a formula for the Poincar\'{e} dual $\alpha$ to the equivariant homology class $i_*([Y]) \in H_*^S(X)$.  (By abuse of notation, we will generally denote the class $\alpha$ by $[Y]$.)  By equivariant localization, this class     is determined by knowing $\alpha|_{wB}$ for each $w \in W$.  Suppose that $wB \in Y$.  Denote by $j_w$ the inclusion of $wB$ into $Y$, and by $i_w$ the inclusion of $wB$ into $X$, so that $i_w = i \circ j_w$.  Then in $H_*^S(X)$, we have the following:
\[ i_w^*(i_*([Y]) = (j_w^* \circ i^*)(i_*([Y])) = j_w^*((i^* \circ i_*)([Y])) =  \]
\[ j_w^*(c_d^S(N_Y X) \cap [Y]) = c_d^S(N_Y X|_{wB}) \cap j_w^*([Y]) = c_d^S(N_Y X|_{wB}) \cap [wB], \]
where $d$ is the codimension of $Y$ in $X$.  Here we have used some basic facts of intersection theory regarding pushforwards and pullbacks, for which the standard reference is \cite{Fulton-IT}.  Note that we are able to use the self-intersection formula because $Y$ is smooth, and hence $E \times^S Y$ is regularly embedded in $E \times^S X$. 

On the other hand, 
\[ i_w^*(i_*([Y])) = i_w^*(\alpha \cap [X]) = \alpha|_{wB} \cap i_w^*([X]) = \alpha|_{wB} \cap [wB]. \]
Then in $H_S^*(X)$, we have
\[ \alpha|_{wB} = c_d^S(N_Y X|_{wB}). \]

Thus computing the restriction of the class $\alpha$ at each $S$-fixed point amounts to computing $c_d^S(N_Y X|_{wB}) \in H_S^*(\{\text{pt.}\}) \cong \C[X_1,\hdots,X_r]$.  We want to compute this Chern class explicitly, as a polynomial in the $X_i$.  Note that the $S$-equivariant bundle $N_Y X|_{wB}$ is simply a representation of the torus $S$, and its top Chern class is the product of the weights of this representation.  We now compute these weights.

The $S$-module $N_Y X|_{wB}$ is simply $T_w X / T_w Y$, so we determine the weights of $S$ on $T_w X$ and $T_w Y$, then remove the weights of $T_w Y$ from those of $T_w X$.  It is standard that 
\[ T_w X = \frg / \text{Ad}(w)(\frb). \]
Since $B$ has been taken to correspond to the negative roots, the weights of $S$ on $T_w X$ are the restrictions of the following weights of $T$ on $T_w X$:
\[ \Phi \setminus w \Phi^- = w \Phi^+. \]

A similar computation can be made for $T_w Y$.  We know that
\[ T_w Y = \frk / (\frk \cap \text{Ad}(w)(\frb)), \]
so the weights of $S$ on $T_w Y$ are as follows:
\[ \Phi_K \setminus (\Phi_K \cap w\Phi^-), \]
where $\Phi_K$ denotes the roots of $K$.  Subtracting this set of weights from those on $T_w X$, we conclude the following:

\begin{prop}\label{prop:restriction-of-closed-orbit}
The weights of $S$ on $N_Y X|_{wB}$ are $\rho(w\Phi^+) \setminus (\rho(w\Phi^+) \cap \Phi_K)$, where $\rho$ denotes restriction $\frt^* \rightarrow \frs^*$.
\end{prop}

Knowing this, the goal is then to describe the closed orbits, as well as the $S$-fixed points contained in each.  The number of closed orbits is known in each of the examples we consider in this paper, see \cite{Richardson-Springer-90,Richardson-Springer-92}.  As explained in \cite[\S 1.3]{Wyser-Thesis}, given a concrete realization of $G$ and $K$ and using basic results from those same references, it is also easy to determine a fixed-point contained in each.

Given that $K \cdot wB$ is a closed orbit, the remaining $S$-fixed points contained in that orbit are of the form $w'w$, with $w' \in W_K$, viewed as an element of $W$ via the inclusion of Weyl groups $W_K \hookrightarrow W$.  Note that it is not completely obvious that $W_K$ is a subgroup of $W$ in the event that $S \subsetneq T$ (as is the case in the examples we consider here), since it is not \textit{a priori} clear that $N_K(S)$ is a subgroup of $N_G(T)$.  That it is follows from that fact that $T$ can be recovered as $Z_G(S)$, the centralizer of $S$ in $G$ (see \cite{Springer-85,Brion-99}).  Since any element of $G$ normalizing $S$ must also normalize $Z_G(S) = T$, we have an inclusion $N_K(S) \subset N_G(T)$.  This gives a map $W_K = N_K(S)/S \rightarrow N_G(T)/T = W$ defined by $nS \mapsto nT$.  The kernel of this map is $\{nS \mid n \in N_K(S) \cap T\}$.  Since $S = K \cap T$, the group $N_K(S) \cap T$ is simply $S$:
\[ N_K(S) \cap T = N_K(S) \cap (T \cap K) = N_K(S) \cap S = S. \]
Thus the kernel of the map $W_K \rightarrow W$ is $\{1\}$, and so it is an inclusion.

\subsection{Other Orbits}\label{ssec:other_orbits}
As alluded to in the previous section, the idea is to compute formulas for classes of all $K$-orbit closures using formulas for the closed orbits as a starting point for applying divided difference operators.  This works because the closed orbits are minimal with respect to the ``weak order" on $K \backslash G/B$ (\cite[Theorem 4.6]{Richardson-Springer-90}).  We now describe the weak ordering, and how divided difference operators enter the picture.  Let $\ga \in \Delta$ be a simple root, and let $P_{\ga}$ be the minimal parabolic subgroup of $G$ of type $\ga$ containing $B$.  Consider the canonical map
\[ \pi_{\ga}: G/B \rightarrow G/P_{\ga}. \]
This is a $\mathbb{P}^1$-bundle.  Letting $Q \in K \backslash G/B$ be given, consider the set $Z_{\ga}(Q):=\pi_{\ga}^{-1}(\pi_{\ga}(Q))$.  The map $\pi_{\ga}$ is $K$-equivariant, so $Z_{\ga}(Q)$ is $K$-stable.  Assuming $K$ is connected, $Z_{\ga}(Q)$ is also irreducible, so it has a dense $K$-orbit.  In the event that $K$ is disconnected, one sees that the component group of $K$ acts transitively on the irreducible components of $Z_{\ga}(Q)$, and from this it again follows that $Z_{\ga}(Q)$ has a dense $K$-orbit.

If $\text{dim}(\pi_{\ga}(Q)) < \text{dim}(Q)$, then the dense orbit on $Z_{\ga}(Q)$ is $Q$ itself.  However, if $\text{dim}(\pi_{\ga}(Q)) = \text{dim}(Q)$, the dense $K$-orbit will be another orbit $Q'$ of one dimension higher.  In either event, using notation as in \cite{McGovern-Trapa-09}, we make the following definition:
\begin{definition}\label{def:weak_order_dot_notation}
With notation as above, $s_{\ga} \cdot Q$ shall denote the dense $K$-orbit on $Z_{\ga}(Q)$.  In the event that $\ga = \ga_i$ for some chosen ordering of the simple roots, if $s_{\ga_i} \cdot Q = Q' \neq Q$, we will also use the notation $Q <_i Q'$ for brevity.
\end{definition}

\begin{definition}
The partial ordering on $K \backslash G/B$ generated by relations of the form $Q < Q'$ if and only if $Q' = s_{\ga} \cdot Q$ (with $\dim(Q') = \dim(Q) + 1$) for some $\ga \in \Delta$ is referred to as the \textbf{weak closure order}, or simply the \textbf{weak order}.
\end{definition}

Let $Y,Y'$ denote the closures of $Q,Q'$, respectively.  Assume that $Q' = s_{\ga} \cdot Q$, and define an operator $\partial_{\ga}$ on $H_S^*(X)$, known as a ``divided difference operator" or a ``Demazure operator",  as follows:
\[ \partial_{\ga}(f) = \frac{f - s_{\ga}(f)}{\ga}. \]

Let $d$ denote the degree of $\pi_{\ga}|_Y$ over its image.  Using standard facts from intersection theory, along with the fact that $\partial_{\ga} = \pi_{\ga}^* \circ (\pi_{\ga})_*$, it is easy to see that $[Y'] = \frac{1}{d} \partial_{\ga}([Y])$.

Putting all of this together, we see that we can recursively determine formulas for the equivariant classes of all orbit closures given the following data:
\begin{enumerate}
	\item Formulas for classes of the closed orbits.
	\item The weak closure order on $K \backslash G/B$.
	\item For any two orbits $Q,Q'$, with closures $Y, Y'$, and with the property that $Q' = s_{\ga} \cdot Q$, the degree $d$ of $\pi_{\ga}|_Y$ over its image.
\end{enumerate}

In fact, the aforementioned degree $d$ is always either $1$ or $2$, as follows from the exposition of \cite[Section 4]{Richardson-Springer-90}.  Namely, $s_{\ga} \cdot Q \neq Q$ only in cases where $\ga$ is a ``complex" or ``non-compact imaginary" root for the orbit $Q$, and the degree $d$ is $2$ if and only if $\ga$ is ``non-compact imaginary type II".  In all other cases, the degree is $1$.  In our examples here, this can all be boiled down to elementary combinatorics involving only subsets of the Weyl group, so we do not discuss the more general picture here.  The interested reader may see \cite{Richardson-Springer-90,Richardson-Springer-92} for more details.

In \cite{Brion-01}, the graph for the weak order on $K$-orbit closures is endowed with additional data, as follows:  If $Y' = s_{\ga} \cdot Y \neq Y$, then the directed edge originating at $Y$ and terminating at $Y'$ is labelled by the simple root $\ga$, or perhaps by an index $i$ if $\ga = \ga_i$ for some predetermined ordering of the simple roots.  Additionally, if the degree of $\pi_{\ga}|_Y$ is $2$, then this edge is double.  (In other cases, the edge is simple.)  We modify this convention as follows:  Rather than use simple and double edges, in our diagrams we distinguish the degree two covers by blue edges, as opposed to the usual black.  (We do this simply because our weak order graphs were created using GraphViz, which does not, as far as the author can ascertain, have a mechanism for creating a reasonable-looking double edge.  On the other hand, coloring the edges is straightforward.)

\section{Examples}
$G$ will be the special linear group, consisting of determinant-$1$ invertible matrices with complex entries.

For a maximal torus $T$ of $G$, let $Y_i$ denote coordinates on $\frt=\text{Lie}(T)$, so that 
\[ \Phi = \{X_i - X_j \ \vert \ i \neq j \}. \]
We choose the ``standard" positive system 
\[ \Phi^+ = \{X_i - X_j \ \vert \ i < j \}, \]
and let $\Phi^- = -\Phi^+$.  Take $B$ to be the Borel subgroup containing $T$ and whose roots are $\Phi^-$.  (Concretely, we may take $T$ to be the diagonal elements of $G$, and $B$ to be the lower-triangular elements of $G$.  Then $\frt$ is the set of all trace-zero diagonal matrices, and $X_i(\text{diag}(a_1,\hdots,a_n)) = a_i$ for each $i$.)

In this case, the Weyl group $W$ is isomorphic to the symmetric group, and elements of $W$ act on the coordinates $X_i$ by permutation of the indices.

\subsection{$K \cong SO(2n+1,\C)$}
We realize $K = SO(2n+1,\C)$ as the subgroup of $G=SL(2n+1,\C)$ preserving the quadratic form given by the antidiagonal matrix $J = J_{2n+1}$.  That is, $K = G^{\theta}$ where $\theta$ is the involution
\[ \theta(g) = J(g^{-1})^tJ. \]

We remark that in the notation of the introduction, this choice of $K$ corresponds to the real form $G_\R = SL(2n+1,\R)$ of $G$.

This realization of $K$ is in fact conjugate to the ``usual" one, that being the fixed point set of the involution $\theta'(g) = (g^{-1})^t$.  We prefer our choice of realization because we can take a maximal torus $S = K \cap T$ consisting of diagonal elements, and a Borel subgroup $B$ consisting of lower-triangular elements.

The torus $\frs = \text{Lie}(S)$ has the form $\text{diag}(a_1,\hdots,a_n,0,-a_n,\hdots,-a_1)$.  Thus if $X_1,\hdots,X_{2n+1}$ represent coordinates on $\frt$, restricting to $\frs$ we have $\rho(X_{n+1}) = 0$, and $\rho(X_i) = Y_i$, $\rho(X_{2n+2-i}) = -Y_i$ for $i=1,\hdots,n$.

The roots of $K$ are as follows:  
\begin{itemize}
	\item $\pm Y_i$ ($i = 1,\hdots,n$)
	\item $\pm (Y_i + Y_j)$ ($1 \leq i < j \leq n)$
	\item $\pm (Y_i - Y_j)$ ($1 \leq i < j \leq n)$
\end{itemize}

The Weyl group $W_K$ of $K$ should be thought of as consisting of signed permutations of $\{1,\hdots,n\}$ (changing any number of signs).  This is the action of $W_K$ on the coordinates $Y_i \in \frs^*$.  $W_K$ is embedded in $W$ as described in Subsection \ref{ssec:notation}.

\subsubsection{A Formula for the Closed Orbit}
As it turns out, there is a unique closed orbit in this case.  In our chosen realization, it is the orbit $K \cdot 1B$, and, by the straightforward remarks given at the end of Subsection \ref{ssec:closed_orbits}, it contains the $S$-fixed points corresponding to elements of $W_K$, embedded in $S_{2n+1}$ as we have just mentioned.

We give a formula for the $S$-equivariant class of the lone closed orbit.

\begin{prop}\label{prop:formula_for_SO_odd}
Let $Q=K \cdot 1B$ be the closed $K$-orbit of the previous proposition.  Then $[Q]$ is represented by
\[ P(x,y) := (-2)^n \displaystyle\prod_{i=1}^n (x_i + x_{n+1})(x_{n+1} + x_{2n+2-i})\displaystyle\prod_{1 \leq i < j \leq n}(x_i + x_j)(x_i + x_{2n+2-j}). \]
\end{prop}
\begin{proof}
We apply Proposition \ref{prop:restriction-of-closed-orbit} to determine the restriction $[Q]|_w$ at a fixed point $w \in Q$.  To compute the set $\rho(w \Phi^+)$, we determine the restrictions of the positive roots $\Phi^+$ to $\frs$, then apply the signed permutation corresponding to $w$ to that set of weights.  (The result is the same as if we viewed $w$ as a signed element of $S_{2n+1}$, applied that permutation to the elements of $\Phi^+$, and \textit{then} restricted the resulting roots to $\frs$.)

Restricting the positive roots $\{X_i - X_j \ \vert \ 1 \leq i < j \leq 2n+1 \}$ to $\frs$, we get the following set of weights:

\begin{enumerate}
	\item $Y_i - Y_j$, $1 \leq i < j \leq n$, each with multiplicity 2 (one is the restriction of $X_i - X_j$, the other the restriction of $X_{2n+2-j} - X_{2n+2-i}$)
	\item $Y_i + Y_j$, $1 \leq i < j \leq n$, each with multiplicity 2 (one is the restriction of $X_i - X_{2n+2-j}$, the other the restriction of $X_j - X_{2n+2-i}$)
	\item $Y_i$, $1 \leq i \leq n$, each with multiplicity 2 (one is the restriction of $X_i - X_{n+1}$, the other the restriction of $X_{n+1} - X_{2n+2-i}$)
	\item $2Y_i$, $1 \leq i \leq n$, each with multiplicity 1 (the restriction of $X_i- X_{2n+2-i}$)
\end{enumerate}

Now, consider applying a signed permutation $w$ to this set of weights.  The resulting set of weights will be

\begin{enumerate}
	\item For each $i,j$ ($1 \leq i < j \leq n$), either $Y_i - Y_j$ or $-(Y_i - Y_j)$, occurring with multiplicity 2 (these weights come from applying $w$ to weights of either type (1) or (2) above);
	\item For each $i,j$ ($1 \leq i < j \leq n$), either $Y_i + Y_j$ or $-(Y_i+Y_j)$, occurring with multiplicity 2 (these weights also come from applying $w$ to weights of either type (1) or (2) above);
	\item For each $i$ ($1 \leq i \leq n$), either $Y_i$ or $-Y_i$, occurring with multiplicity 2 (these weights come from applying $w$ to weights of type (3) above);
	\item For each $i$ ($1 \leq i \leq n$), either $2Y_i$ or $-2Y_i$, ocurring with multiplicity 1 (these weights come from applying $w$ to weights of type (4) above).
\end{enumerate}

Discarding roots of $K$, we are left with the following weights:

\begin{enumerate}
	\item For each $i,j$ ($1 \leq i < j \leq n$), either $Y_i - Y_j$ or $-(Y_i-Y_j)$, occurring with multiplicity 1;
	\item For each $i,j$ ($1 \leq i < j \leq n$), either $Y_i + Y_j$ or $-(Y_i+Y_j)$, occurring with multiplicity 1;
	\item For each $i$ ($1 \leq i \leq n$), either $Y_i$ or $-Y_i$, occurring with multiplicity 1;
	\item For each $i$ ($1 \leq i \leq n$), either $2Y_i$ or $-2Y_i$, occurring with multiplicity 1.
\end{enumerate}

It is clear that the number of weights of the form $-Y_i$ and the number of weights of the form $-2Y_i$ are the same, so weights of those two forms account for an even number of negative signs.  So in computing the restriction, to get the sign right, we need only concern ourselves with the signs of the weights of types (1) and (2) above.

We claim that the number of $Y_i \pm Y_j$ ($i < j$) occurring with a negative sign is congruent mod $2$ to $l(|w|)$.  (Cf. Subsection \ref{ssec:notation} for this notation.)  Indeed, suppose first that $|w|$ does not invert $i$ and $j$, so that $k = |w(i)| < |w(j)| = l$.  Then there are four possibilities:

\begin{enumerate}
	\item $w(i)$, $w(j)$ are both positive.  In this case, $Y_{w(i)} + Y_{w(j)} = Y_k + Y_l$, and $Y_{w(i)} - Y_{w(j)} = Y_k - Y_l$.  Neither of these is a negative root.
	\item $w(i)$ is negative, and $w(j)$ is positive.  Then $Y_{w(i)} + Y_{w(j)} = -(Y_k - Y_l)$, and $Y_{w(i)} - Y_{w(j)} = -(Y_k + Y_l)$.  Both of these are negative roots.
	\item $w(i)$ is positive, and $w(j)$ is negative.  Then $Y_{w(i)} + Y_{w(j)} = Y_k - Y_l$, and $Y_{w(i)} - Y_{w(j)} = Y_k + Y_l$.  Neither of these is a negative root.
	\item $w(i)$, $w(j)$ are both negative.  Then $Y_{w(i)} + Y_{w(j)} = -(Y_k + Y_l)$, and $Y_{w(i)} - Y_{w(j)} = -(Y_k - Y_l)$.  Both of these are negative roots.
\end{enumerate}

All this is to say that if $|w|$ does not invert $i$ and $j$, then this accounts for an even number of negative signs occurring in the restriction.  On the other hand, if $|w|$ does invert $i$ and $j$, so that $k = |w(j)| < |w(i)| = l$, then again there are four possibilities:

\begin{enumerate}
	\item $w(i)$, $w(j)$ are both positive.  In this case, $Y_{w(i)} + Y_{w(j)} = Y_k + Y_l$, and $Y_{w(i)} - Y_{w(j)} = -(Y_k - Y_l)$.  One of these is a negative root.
	\item $w(i)$ is negative, and $w(j)$ is positive.  Then $Y_{w(i)} + Y_{w(j)} = Y_k - Y_l$, and $Y_{w(i)} - Y_{w(j)} = -(Y_k + Y_l)$.  One of these is a negative root.
	\item $w(i)$ is positive, and $w(j)$ is negative.  Then $Y_{w(i)} + Y_{w(j)} = -(Y_k - Y_l)$, and $Y_{w(i)} - Y_{w(j)} = Y_k + Y_l$.  One of these is a negative root.
	\item $w(i)$, $w(j)$ are both negative.  Then $Y_{w(i)} + Y_{w(j)} = -(Y_k + Y_l)$, and $Y_{w(i)} - Y_{w(j)} = Y_k - Y_l$.  One of these is a negative root.
\end{enumerate}

The upshot is that if $w \in Q$ is an $S$-fixed point, then

\[ [Q]|_w = F(Y) := (-1)^{l(|w|)} 2^n \displaystyle\prod_{i = 1}^n Y_i^2 \displaystyle\prod_{1 \leq i < j \leq n} (Y_i + Y_j)(Y_i-Y_j). \]

So we seek a polynomial in $x_1,\hdots,x_{2n+1},y_1,\hdots,y_n$, say $p$, with the property that

\[ p(\rho(wX),Y) = 
\begin{cases}
F(Y) & \text{ if $w \in W_K$} \\
0 & \text{ otherwise.}
\end{cases} \]

It is straightforward to check that $P(x,y)$ has these properties.  Indeed, suppose first that $w \in W_K$.  (We should think of $w$ here as a signed element of $S_{2n+1}$, since this is how $w$ acts on the $X_i$.)  Consider first the factors $x_i + x_{n+1}$ and $x_{n+1}+x_{2n+2-i}$ for $i=1,\hdots,n$.  Supposing $w(i) \leq n$, $X_i + X_{n+1}$ gives $X_{w(i)} + X_{n+1}$, which restricts to $Y_{w(i)} + 0 = Y_{w(i)}$.  On the other hand, $X_{n+1} + X_{2n+2-i}$ gives $X_{n+1} + X_{w(2n+2-i)}$, which restricts to $0 - Y_{w(i)} = -Y_{w(i)}$, so the product $(x_i + x_{n+1})(x_{n+1}+x_{2n+2-i})$ restricts to $-Y_{w(i)}^2$.  If $w(i) > n+1$, then the product of these two terms restricts to $-Y_{2n+2-w(i)}^2$, with the negative term coming from $x_i + x_{n+1}$, and the positive term coming from the $x_{n+1} + x_{2n+2-i}$.  As $i$ runs from $1$ to $n$, the product of all these terms restricts to $(-1)^n \prod_{i=1}^n Y_i^2$.  This explains the factor of $(-2)^n$ in our formula, as opposed to just $2^n$.  The $(-1)^n$ is to account for a possible sign flip coming from terms of this type.  So the terms $(-2)^n \prod_{i=1}^n (x_i + x_{n+1})(x_{n+1} + x_{2n+2-i})$ of our putative formula contribute the $2^n \prod_{i=1}^n Y_i^2$ portion of the required restriction.

Next, consider the terms $x_i + x_j$ and $x_i + x_{2n+2-j}$.  Applying $w$ and restricting, these give (up to sign) all required terms of the form $Y_i + Y_j$ and $Y_i - Y_j$ ($i<j$).  Writing each such term as either $+1$ or $-1$ times a positive root by factoring out negative signs as necessary, we effectively introduce the sign of $(-1)^{l(|w|)}$, as required.

On the other hand, given any $w \notin W_K$ (i.e. a \textit{non-signed} element of $S_{2n+1}$), there are two possibilities:

\textit{Case 1:  $w$ does not fix $n+1$}

In this case, $w$ moves $n+1$ to some $i$ such that $1 \leq i \leq 2n+1$, and $i \neq n+1$.  Let $j = w^{-1}(2n+2-i)$.  (Note, of course, that $j \neq n+1$.)  Applying $w$ to $x_j + x_{n+1}$, we get $X_{2n+2-i} + X_i$, which restricts to $0$.

\textit{Case 2:  $w$ fixes $n+1$}

In this case, $w(2n+2-i) \neq 2n+2-w(i)$ for some $1 \leq i \leq n$.  Let $j = 2n+2-w(i)$, and let $k = w^{-1}(j)$.  Clearly, $k \neq i$, $2n+2-i$, or $n+1$, so the factor $x_i + x_k$ appears in $P$.  Applying $w$ to this factor gives $X_{w(i)} + X_{2n+2-w(i)}$, which restricts to zero.

We see that in either case, applying $w$ then restricting kills one of the factors appearing in $P$, and so the result is zero for any $w \notin W_K$, as desired.  This completes the proof.
\end{proof}

\begin{remark}
An alternate representative of $[Q]$ is
\[ P(x,y) := (-2)^n \displaystyle\prod_{i=1}^n (x_{n+1}+y_i)(x_{n+1}-y_i)\displaystyle\prod_{1 \leq i < j \leq n}(x_i + x_j)(x_i + x_{2n+2-j}). \]
Indeed, this was the first representative discovered by the author.  However, the representative of the previous proposition is preferable from our perspective, essentially because a formula involving only the $x_i$ will pull back to a Chern class formula for the class of a certain degeneracy locus.  It is not clear that the representative involving the $y_i$ should have such an interpretation.
\end{remark}

\subsubsection{Parametrization of $K \backslash G/B$ and the Weak Order}\label{ssec:kgb_param_so_odd}
As described in \cite[Examples 10.2,10.3]{Richardson-Springer-90}, the $K$-orbits in this case are in bijection with the set of ``twisted involutions".  This is a subset of $W$, defined in \cite{Richardson-Springer-90}.  In this particular case, the twisted involutions turn out to be in bijection with the ordinary involutions of $W$.  (The passage from twisted involutions to honest involutions is accomplished by left-multiplication by the long element $w_0$.)  Thus the $K$-orbits on $G/B$ are parametrized by involutions of $W$.

Relative to this parametrization, the lone closed orbit $Q$ corresponds to the long element $w_0=(1,2n+1)(2,2n)\hdots(n,n+2)$, and the weak order graph is generated from this starting point by the following rules:  Given an involution $b \in W$, with $Q_b$ the corresponding $K$-orbit,
\begin{enumerate}
	\item If $l(s_ib) > l(b)$, then $s_i \cdot Q_b = Q_b$.
	\item Else, if $s_ibs_i \neq b$, then $Q_b <_i Q_{s_ibs_i}$, and the edge in the weak order graph is black.
	\item Else, $Q_b <_i Q_{s_ib}$, and the edge in the weak order graph is blue.
\end{enumerate}

This description of the weak order follows from some basic results of \cite{Richardson-Springer-90}, combined with straightforward combinatorial arguments.  We omit the details here, but the interested reader can find them in \cite[\S 2.2.2]{Wyser-Thesis}.

The parametrization of $K \backslash G/B$ by involutions is convenient because an involution $b \in W$ encodes a linear algebraic description of the orbit corresponding to $b$ in a straightforward way.  Namely, given an involution $b$, define, for any $i$ and $j$,
\[ r_b(i,j) := \#\{ k \leq i \ \vert \ b(k) \leq j \}. \]
Let $V = \C^{2n+1}$, and let $\gamma: V \otimes V \rightarrow \C$ denote the orthogonal form with isometry group $K$.  For any flag $F_{\bullet} = (F_1 \subset \hdots \subset F_{2n+1}) \in X$, denote by $\gamma|_{F_i \times F_j}$ the restriction of $\gamma$ to pairs of the form $(v,w)$ with $v \in F_i$ and $w \in F_j$.  The claim is that if $b \in W$ is an involution, then the set
\[ Q_b := \{ F_{\bullet} \in X \mid \text{rank}(\gamma|_{F_i \times F_j}) = r_b(i,j) \text{ for all } i,j \} \]
is a $K$-orbit on $G/B$, and that the association $b \mapsto Q_b$ defines a bijection between involutions in $W$ and $K$-orbits.

Since any permutation $w$ is uniquely determined by the set of numbers $r_w(i,j)$, it is clear by definition that the $Q_b$ are mutually disjoint.  It is also clear that each set $Q_b$, if non-empty, is stable under $K$ and hence is at least a union of $K$-orbits.  If we can see that every $Q_b$ is non-empty, it will follow that each must be a \textit{single} $K$-orbit.  Indeed, as mentioned above, we know by the results of \cite{Richardson-Springer-90} that the orbits are in bijection with the involutions of $W$.  If each $Q_b$ is non-empty, then it is impossible for any one of them to be anything other than a single $K$-orbit, for then there would be more $K$-orbits than involutions.

Thus we show that each set $Q_b$ is non-empty by producing an explicit representative satisfying the appropriate rank conditions.  It suffices to produce a basis $\{v_1,\hdots,v_{2n+1}\}$ for $\C^{2n+1}$ such that the matrix for the form $\gamma$ relative to this basis is a monomial matrix (that is, a matrix such that each row and column has exactly one non-zero entry) whose image in $W$ is $b$.  Then we can simply take our flag $F_{\bullet}$ to be $\left\langle v_1,\hdots,v_{2n+1} \right\rangle$.

We choose such a basis as follows.  (Recall that the form $\gamma$ is defined by $\left\langle e_i,e_j \right\rangle = \delta_{i,2n+2-j}$.)  First, for each $i$ such that $b(i) \neq i$, choose $v_i$ and $v_{b(i)}$ to be $e_k$ and $e_{2n+2-k}$ for some $k \neq n+1$.  (Of course, we should choose a different such $k$ for each such $i$.)   There are an odd number of $i$ such that $b(i) = i$ --- for one such $i$, choose $v_i$ to be $e_{n+1}$, and for all other pairs $i_1,i_2$ of such $i$, choose $v_{i_1}$ to be $e_k + e_{2n+2-k}$ for some $k \neq n+1$ (and not yet used in the first step above), and choose $v_{i_2}$ to be $e_k - e_{2n+2-k}$ for the same $k$.  (We should choose a different such $k$ for each such pair $i_1,i_2$.)

\begin{prop}
With $v_1,\hdots,v_{2n+1}$ defined as above, the flag $\left\langle v_1,\hdots,v_{2n+1} \right\rangle$ lies in $Q_b$.
\end{prop}
\begin{proof}
We first note that the matrix for the form $\gamma$ relative to this basis is indeed a monomial matrix whose image in $W$ is $b$.  This means precisely that for each $i$, $\left\langle v_i,v_j \right\rangle$ is non-zero if and only if $j = b(i)$.

For any $i$ with $b(i) \neq i$, this is clear.  Indeed, $v_i = e_k$ for some $k$, while $v_{b(i)} = e_{2n+2-k}$.  Meanwhile, $e_k$ appears with coefficient $0$ in all other $v_i$ by design.  Since $\left\langle e_i,e_j \right\rangle = \delta_{i,2n+2-j}$, we see that $\left\langle v_i,v_j \right\rangle = \delta_{j,b(i)}$.

Now suppose that $b(i) = i$.  Then either $v_i = e_{n+1}$, or $v_i = e_k  \pm e_{2n+2-k}$ for some $k \neq n+1$ (and not equal to any $k$ used to define $v_i$ with $b(i) \neq i$).  In the former case, we have 
\[ \left\langle v_i,v_{b(i)} \right\rangle = \left\langle v_i,v_i \right\rangle = \left\langle e_{n+1},e_{n+1} \right\rangle = 1, \]
while $\left\langle v_i,v_j \right\rangle = 0$ for any other $j$, since $e_{n+1}$ appears with coefficient $0$ in all other $v_i$.  In the latter case, supposing that $v_i = e_k  + e_{2n+2-k}$, we have
\[ \left\langle v_i,v_{b(i)} \right\rangle = \left\langle v_i,v_i \right\rangle = \left\langle e_k,e_{2n+2-k} \right\rangle + \left\langle e_{2n+2-k},e_k \right\rangle) = 2. \]
If $v_i = e_k  - e_{2n+2-k}$, the corresponding computation shows that $\left\langle v_i,v_{b(i)} \right\rangle = -2$.

For $j \neq b(i)$, either $e_k$ appears with coefficient $0$ in $v_j$ (in which case $\left\langle v_i,v_j \right\rangle = 0$), or $v_j = e_k \mp e_{2n+2-k}$, and in that case,
\[ \left\langle v_i,v_j \right\rangle = \left\langle e_k,e_{2n+2-k} \right\rangle - \left\langle e_{2n+2-k},e_k \right\rangle) = 0. \]

This establishes that the matrix for $\gamma$ relative to the basis $\{v_i\}$ is indeed monomial, with image $b$ in $W$.  Now, note that if $F_{\bullet} = \left\langle v_1,\hdots,v_{2n+1} \right\rangle$, then $\text{rank}(\gamma|_{F_i \times F_j})$ is, by definition, the rank of the upper-left $i \times j$ rectangle of this matrix.  For any monomial matrix with image $b$ in $W$, the rank of the upper-left $i \times j$ rectangle is precisely $r_b(i,j)$.  This proves the claim.
\end{proof}

We illustrate with two examples.  Suppose $n = 2$, so we are dealing with $G = GL(5,\C)$, $K = O(5,\C)$.  First consider the involution $b=(2,4)$.  Since $b$ moves $2$ and $4$, we first choose $v_2 = e_1$ and $v_4 = e_5$.  Since $b$ fixes $1$, $3$, and $5$, we first choose $v_1 = e_3$, then we choose $v_3 = e_2+e_4$ and $v_5 = e_2 - e_4$.  Our ordered basis is thus
\[ \{ e_3,e_1,e_2+e_4,e_5,e_2-e_4 \}. \]
Relative to this ordered basis, the form $\gamma$ has matrix
\[ \begin{pmatrix}
1 & 0 & 0 & 0 & 0 \\
0 & 0 & 0 & 1 & 0 \\
0 & 0 & 2 & 0 & 0 \\
0 & 1 & 0 & 0 & 0 \\
0 & 0 & 0 & 0 & -2 \end{pmatrix}, \]
and one checks that if $F_{\bullet} = \left\langle v_1,\hdots,v_5 \right\rangle$, then the rank conditions specified by $b=(2,4)$ are satisfied.

Next, consider $b=(1,3)(2,5)$.  We first choose $v_1 = e_1$, $v_3 = e_5$, $v_2 = e_2$, and $v_5 = e_4$.  Finally, since $b$ fixes only $4$, we choose $v_4 = e_3$.  So our ordered basis is $\left\langle e_1,e_2,e_5,e_3,e_4 \right\rangle$, and the form $\gamma$, relative to this basis, has matrix
\[ \begin{pmatrix}
0 & 0 & 1 & 0 & 0 \\
0 & 0 & 0 & 0 & 1 \\
1 & 0 & 0 & 0 & 0 \\
0 & 0 & 0 & 1 & 0 \\
0 & 1 & 0 & 0 & 0 \end{pmatrix}. \]
One again checks easily that if $F_{\bullet} = \left\langle v_1,\hdots,v_5 \right\rangle$, the rank conditions encoded by $b = (1,3)(2,5)$ are satisfied.

Using the linear algebraic description of $K$-orbits as sets of flags, it is also easy to describe $K$-orbit \textit{closures} as sets of flags.  This will be useful to us in Section \ref{sec:deg-loci}, when we will realize the $K$-orbit closures as universal cases of degeneracy loci bearing similar linear algebraic descriptions.  Indeed, the result is the following:

\begin{prop}\label{prop:ortho-orbit-closures}
Suppose that $b \in S_{2n+1}$ is an involution, with $Q_b$ the associated $K$-orbit.  Then
\[ \overline{Q_b} = \{F_{\bullet} \mid \text{rank}(\gamma|_{F_i \times F_j}) \leq r_b(i,j) \text{ for all } i,j \} \]
\end{prop}
\begin{proof}
As we have mentioned, the orbits in this case are parametrized by the twisted involutions of $W$, and it is explained in \cite{Richardson-Springer-92} (with the proof appearing in \cite{Richardson-Springer-94}) that in such cases, the closure order on $K$-orbits is given precisely by the induced Bruhat order on twisted involutions.  Passing from twisted involutions to honest involutions by left-multiplication by $w_0$ inverts this order, so that when $K$-orbits are parametrized by involutions, their closure order is given precisely by the \textit{reverse} Bruhat order on these involutions.  The claim now follows from the definition of the Bruhat order given in \cite[\S 10.5]{Fulton-YoungTableaux} in terms of the rank numbers $r_b(i,j)$.
\end{proof}

\subsubsection{Example}
We now work out a very small example in detail, the case $n = 1$.  (So $G = GL(3,\C)$, $K = O(3,\C)$.)  Here, there are 4 involutions, and hence 4 orbits.

We start from the minimal element $w_0 = (1,3)$, and work our way up as described in the previous subsection.  Since
\[ s_1 w_0 s_1 = (1,2)(1,3)(1,2) = (2,3) \neq w_0, \]
we have $Q_{w_0} <_1 Q_{(2,3)}$, and the edge is black.

Similarly,
\[ s_2 w_0 s_2 = (2,3)(1,3)(2,3) = (1,2) \neq w_0, \]
so $Q_{w_0} <_2 Q_{(1,2)}$, and again the edge is black.

Now, we move up to the orbits corresponding to $(1,2)$ and $(2,3)$.  Start with $(2,3) = s_2$.  Since $l(s_1 s_2) > l(s_2)$, $s_1 \cdot Q_{(2,3)} = Q_{(2,3)}$.  So we check $s_2$.  Since
\[ s_2 s_2 s_2 = s_2, \]
we have $Q_{(2,3)} <_2 Q_{s_2 s_2} = Q_{1}$, and in this case the edge is blue.

The situation with $(1,2) = s_1$ is identical, with $s_2$ above replaced by $s_1$, so that $Q_{(1,2)} <_1 Q_1$, and the edge is blue.  The weak order graph appears as Figure \ref{fig:type-a-orthogonal-1} of the appendix.

With this complete, we now determine formulas for the $S$-equivariant classes of all orbit closures.  By Proposition \ref{prop:formula_for_SO_odd} above, the class of the closed orbit corresponding to $w_0$ is given by the formula $[Q] = -2(x_1+x_2)(x_2+x_3)$.  The class $[Y_{(2,3)}]$ is given by
\[ [Y_{(2,3)}] = \partial_1([Q]), \]
and since 
\[ \partial_1(f(x_1,x_2,x_3,y)) = \dfrac{f - f(x_2,x_1,x_3,y)}{x_1-x_2}, \]
we have $[Y_{(2,3)}] = 2(x_1+x_2)$.  Similarly, $[Y_{(1,2)}] = \partial_2([Q]) = -2(x_2+x_3)$.

Finally, we can compute $[Y_{id}]$ either as $\frac{1}{2}\partial_2([Y_{(2,3)}])$, or as $\frac{1}{2}\partial_1([Y_{(1,2)}])$.  Using either formula, we get that $[Y_{id}] = 1$, as expected.

The results are summarized in Table \ref{tab:type-a-so3} of the appendix.  The weak order graph and the list of formulas for the larger case $n = 5$ appear in Figure \ref{fig:type-a-orthogonal-2} and Table \ref{tab:type-a-so5}.  (In that case, there are 26 orbits.)

\subsection{$K \cong SO(2n,\C)$}
We now treat the case of the even special orthogonal group.

As before, we realize $K$ as the subgroup of $SL(2n,\C)$ preserving the orthogonal form given by the antidiagonal matrix $J=J_{2n}$.  

Here, if $X_i$ ($i=1,\hdots,2n$) are coordinates on $\frt$, restriction to $\frs$ is given by $\rho(X_i) = Y_i$ and $\rho(X_{2n+1-i}) = -Y_i$ for each $i = 1,\hdots,n$.

The roots of $K$ are 
\[ \Phi_K  = \{\pm(Y_i \pm Y_j) \mid i<j\}. \]

In this case, the Weyl group $W_K$ of $K$ acts on torus characters by signed permutations which change an \textit{even} number of signs.  The inclusion of $W_K$ into $W$ described in Subsection \ref{ssec:notation} thus has the further property that
\begin{equation}
	\sigma(i) > n \text{ for an even number of } i=1,\hdots,n.
\end{equation}

\subsubsection{Formulas for the Closed Orbits}
There are $2$ closed orbits in this case \cite[Example 10.3]{Richardson-Springer-90}.  In our chosen realization, these are $Q_1$, the orbit $K \cdot 1B$, and $Q_2$, the orbit $K \cdot s_nB$, with $s_n$ the simple transposition $(n,n+1)$.  Fixed points in the orbit $Q_1$ correspond to the elements of $W_K$, i.e. the signed permutations of $\{1,\hdots,n\}$ changing an even number of signs, embedded in $S_{2n}$ as just described above.  Fixed points in the orbit $Q_2$ correspond to $(n,-n) \cdot W_K$, where $(n,-n)$ denotes the signed permutation of $\{1,\hdots,n\}$ which interchanges $n$ with $-n$.  (Note that $s_n \in S_{2n}$ is the image of $(n,-n)$ under our preferred embedding of signed permutations into $W$.)  Thus $Q_2$ contains the fixed points of $S_{2n}$ which correspond to signed permutations of $\{1,\hdots,n\}$ changing an \textit{odd} number of signs.

We give formulas for the $S$-equivariant classes of $Q_1$ and $Q_2$:

\begin{prop}\label{prop:formula_for_SO_even}
With $Q_1$ and $Q_2$ as in the previous proposition, $[Q_1]$ is represented by the polynomial $P_1(x,y)$, and $[Q_2]$ by the polynomial $P_2(x,y)$, where
\[ P_1(x,y) = 2^{n-1} (x_1 \hdots x_n + y_1 \hdots y_n) \displaystyle\prod_{1 \leq i < j \leq n}(x_i + x_j)(x_i + x_{2n+1-j}); \]
and
\[ P_2(x,y) = 2^{n-1} (x_1 \hdots x_n - y_1 \hdots y_n) \displaystyle\prod_{1 \leq i < j \leq n}(x_i + x_j)(x_i + x_{2n+1-j}). \]
\end{prop}
\begin{proof}
We demonstrate the correctness of the formula for $[Q_1]$. The argument is similar to that given in the previous case for the lone closed orbit of the odd orthogonal group.

As stated, $Q_1$ consists of those $S$-fixed points corresponding to elements of $W_K$ --- that is, signed permutations with an even number of sign changes.  Take $w \in Q_1$ to be such a fixed point.  We use Proposition \ref{prop:restriction-of-closed-orbit} to compute the restriction $[Q_1]|_w$.  As in the previous example, we first determine the restriction of the positive roots $\Phi^+$ to $\frs$, then apply the signed permutation $w$ to that set of weights.

Restricting the positive roots $\{X_i - X_j \ \vert \ 1 \leq i < j \leq 2n\}$ to $\frs$, we get the following set of weights:

\begin{enumerate}
	\item $Y_i - Y_j$, $1 \leq i < j \leq n$, each with multiplicity 2 (one is the restriction of $X_i - X_j$, the other the restriction of $X_{2n+1-j} - X_{2n+1-i}$)
	\item $Y_i + Y_j$, $1 \leq i < j \leq n$, each with multiplicity 2 (one is the restriction of $X_i - X_{2n+1-j}$, the other the restriction of $X_j - X_{2n+1-i}$)
	\item $2Y_i$, $1 \leq i \leq n$, each with multiplicity 1 (the restriction of $X_i - X_{2n+1-i}$)
\end{enumerate}

Now, consider applying a signed permutation $w \in W_K$ to this set of weights.  The resulting set of weights will be

\begin{enumerate}
	\item $\pm (Y_i - Y_j)$, $1 \leq i < j \leq n$, each occurring with either a plus or minus sign, and with multiplicity 2 (these weights come from applying $w$ to weights of either type (1) or (2) above)
	\item $\pm (Y_i + Y_j)$, $1 \leq i < j \leq n$, each occurring with either a plus or minus sign, and with multiplicity 2 (these weights also come from applying $w$ to weights of either type (1) or (2) above)
	\item $\pm 2Y_i$, $1 \leq i \leq n$, each ocurring with either a plus or minus sign, and with multiplicity 1 (these weights come from applying $w$ to weights of type (3) above)
\end{enumerate}

Subtracting roots of $K$, we are left with the following weights:

\begin{enumerate}
	\item $\pm (Y_i - Y_j)$, $1 \leq i < j \leq n$, each occurring with either a plus or minus sign, and with multiplicity 1
	\item $\pm (Y_i + Y_j)$, $1 \leq i < j \leq n$, each occurring with either a plus or minus sign, and with multiplicity 1
	\item $\pm 2Y_i$, $1 \leq i \leq n$, each occurring with either a plus or minus sign, and with multiplicity 1
\end{enumerate}

The number of weights of the form $-2Y_i$ is even, since $w$ changes an even number of signs.  So in computing the restriction, to get the sign right, we need only concern ourselves with the signs of the weights of types (1) and (2) above.

We may argue just as in the proof of Proposition \ref{prop:formula_for_SO_odd} that the number of $Y_i \pm Y_j$ ($i < j$) occurring with a negative sign is congruent mod $2$ to $l(|w|)$.  As such, if $w \in Q_1$ is an $S$-fixed point, then

\[ [Q_1]|_w = F(Y) := (-1)^{l(|w|)} 2^n Y_1 \hdots Y_n \displaystyle\prod_{1 \leq i < j \leq n} (Y_i + Y_j)(Y_i-Y_j). \]

So we seek a polynomial in $x_1,\hdots,x_{2n},y_1,\hdots,y_n$, say $f$, with the property that

\[ f(\rho(wX),Y) = 
\begin{cases}
F(Y) & \text{ if $w \in W_K$} \\
0 & \text{ otherwise.}
\end{cases} \]

It is straightforward to verify that $P_1$ has these properties.  Indeed, first take $w \in W_K$.  Then applying $w$ to the term $y_1 \hdots y_n + x_1 \hdots x_n$ gives $2Y_1 \hdots Y_n$, since $w$ permutes the $X_i$ with an even number of sign changes, and each restricts to the corresponding $Y_i$.  Multiplying this by $2^{n-1}$ gives us the $2^n Y_1 \hdots Y_n$ part of $F$.  The terms $x_i + x_j$ and $x_i + x_{2n+1-j}$ give, up to sign, all terms of the form $Y_i + Y_j$ and $Y_i - Y_j$ ($i<j$).  Rewriting each such term as either $+1$ or $-1$ times a positive root by factoring out negative signs as necessary, we effectively introduce the sign of $(-1)^{l(|w|)}$, as required.

On the other hand, if $w \notin W_K$, then there are two possibilities:

\textit{Case 1:  $w$ is a signed element of $S_{2n}$ corresponding to a signed permutation with an \textit{odd} number of sign changes.}

In this case, $w$ clearly kills the term $y_1 \hdots y_n + x_1 \hdots x_n$, and hence $f(\rho(wX),Y) = 0$.

\textit{Case 2:  $w$ is not a signed element of $S_{2n}$.}

In this case, then $w(2n+1-i) \neq 2n+1-w(i)$ for some $1 \leq i \leq n$.  Let $j = 2n+1-w(i)$, and let $k = w^{-1}(j)$.  Clearly, $k \neq i$ or $2n+1-i$.  So the factor $x_i + x_k$ appears in $P_1$.  Applying $w$ to this factor gives $X_{w(i)} + X_{2n+1-w(i)}$, which then restricts to zero.

We see that in either case, $f(\rho(wX),Y) = 0$.  This proves that $P_1(x,y)$ represents $[Q_1]$.

The verification of the formula for $[Q_2]$ is very similar, and so is omitted.
\end{proof}

\begin{remark}
Note that the representatives for $[Q_1]$ and $[Q_2]$ involve both the $x$ and $y$ variables.  Unlike the odd case, there don't seem to be representatives involving only the $x$-variables (at least not that the author was able to find).  However, note that if we consider the lone closed orbit of $O(2n,\C)$ on $X$ (the union of $Q_1$ and $Q_2$), its class (being the sum of $[Q_1]$ and $[Q_2]$) involves only the $x$-variables.  This reflects the fact that the fundamental classes of degeneracy loci parametrized by $O(2n,\C)$-orbit closures are expressible in the Chern classes of a flag of vector subbundles of a given vector bundle $V$ over a variety $X$.  By contrast, the fundamental classes of the \textit{irreducible components} of such loci, parametrized by $SO(2n,\C)$-orbit closures, are only expressible in these Chern classes together with the \textit{Euler class} of the bundle $V$, see \cite{Edidin-Graham-95}.  See Section \ref{sec:deg-loci} for more details.
\end{remark}

\subsubsection{Parametrization of the Orbits and the Weak Order}\label{ssec:so2n_param}
Again we refer to \cite[Examples 10.2,10.3]{Richardson-Springer-90}.  Although we took $K$ to be $SO(2n+1,\C)$ in Subsection \ref{ssec:kgb_param_so_odd}, the orbits of $O(2n+1,\C)$ on $GL(2n+1,\C)/B$ is identical to the description given there.  Indeed, when one deals with the odd orthogonal group, the element $-1$ lies in the non-identity component, so that one can pass from one component of this group to the other by an element which acts trivially on $GL(2n+1,\C)/B$.  If one thinks of $K$ being the full orthogonal group instead of the special orthogonal group, then the parametrization described in Subsection \ref{ssec:kgb_param_so_odd} applies equally well to the even case.  That is, $O(2n,\C)$-orbits on $GL(2n,\C)/B$ are again parametrized by involutions of $S_{2n}$, the weak order is described the same way, and the orbits (and their closures) bear the same linear algebraic descriptions.  However, when one considers the $SO(2n,\C)$-orbits on $G/B$, things are a bit more complicated.  Some of the $O(2n,\C)$-orbits coincide with a single $SO(2n,\C)$-orbit, while others split as a union of two distinct $SO(2n,\C)$-orbits.  As described in \cite[Examples 10.2,10.3]{Richardson-Springer-90}, the precise result is as follows:  If $b \in S_{2n}$ is an involution, and $Q_b$ is the corresponding $O(2n,\C)$-orbit on the flag variety, then
\begin{enumerate}
	\item $Q_b$ is a single $SO(2n,\C)$-orbit if $b$ has a fixed point.
	\item $Q_b$ is the union of two distinct $SO(2n,\C)$-orbits if $b$ is fixed point-free.
\end{enumerate}

If $b$ is an involution with fixed points, then one can determine a representative of the $SO(2n,\C)$-orbit $\caO_b$ just as described in Subsection \ref{ssec:kgb_param_so_odd}.  If $b$ is fixed point-free, then one can determine a representative of the $O(2n,\C)$-orbit corresponding to $b$ using the same procedure.  This gives a representative of one of the two $SO(2n,\C)$-orbits which correspond to $b$.  Note that this representative is always an $S$-fixed flag, corresponding to a permutation in $S_{2n}$.  To get a representative of the other $SO(2n,\C)$-orbit corresponding to $b$, one can multiply this permutation by the transposition $(n,n+1)$ and take the $S$-fixed flag corresponding to the resulting element of $S_{2n}$.

The two closed orbits are particular examples of this.  Indeed, the closed orbits are the two components of the $O(2n,\C)$-orbit corresponding to the involution $w_0$, which is fixed point-free.  To get a representative of one component, one follows the procedure of Subsection \ref{ssec:kgb_param_so_odd} to obtain the standard flag $\left\langle e_1,\hdots,e_{2n} \right\rangle$.  Then, to get a representative of the other component, we apply the permutation $(n,n+1)$ to obtain $\left\langle e_1,\hdots,e_{n-1},e_{n+1},e_n,e_{n+2},\hdots,e_{2n} \right\rangle$.

The weak closure order on $SO(2n,\C)$-orbits, as well as whether edges of the weak order graph are black or blue, require a bit more care to get right when dealing with orbits which are components of $O(2n,\C)$-orbits.  Given two $O(2n,\C)$-orbits $Q_1$ and $Q_2$, with $Q_1 <_i Q_2$, supposing that either orbit (or both) splits as a union of two $SO(2n,\C)$-orbits, how does one describe the weak order on the components?

Note that there are two possible ways this can occur:  Either $Q_1$ and $Q_2$ \textit{both} split, or $Q_1$ splits and $Q_2$ does not.  Each possibility can occur, as we see in the case $n = 2$.  Indeed, when considering $O(4,\C)$-orbits, parametrized by involutions, we have $w_0 <_1 (1,3)(2,4)$, each of which is fixed point-free.  Thus both of these orbits split.  We also have $w_0 <_2 (1,4)$, and $(1,4)$ has fixed points, so it does not split.  The third ``possibility", where $Q_1$ does not split while $Q_2$ does, is clearly not possible either from a geometric or a combinatorial standpoint.  Indeed, it cannot happen that two different components of $Q_2$ are \textit{both} dense in $\pi_{\ga}^{-1}(\pi_{\ga}(Q_1))$ for $\ga = \ga_i \in \Delta$.  This is reflected combinatorially by the the fact that if $b \in W$ is an involution with fixed points, and if $l(s_ib) < l(b)$, then both of the following must hold:
\begin{enumerate}
	\item $s_i b s_i$ has fixed points.  Indeed, if $b$ fixes any value other than $i$ or $i+1$, then $s_i b s_i$ fixes that same value.  Otherwise, if $b(i) = i$, then $s_i b s_i(i+1) = i+1$, and if $b(i+1) = i+1$, then $s_i b s_i(i) = i$.
	\item If $s_ibs_i = b$, then $s_i b$ also has fixed points.  Indeed, if $s_i b s_i = b$, then $b$ must preserve the set $\{i,i+1\}$, as well as its complement.  If $b$ fails to fix any value other than $i$ and $i+1$, then it must fix both $i$ and $i+1$, since $b$ is assumed to have fixed points.  But in this case, we have $l(s_ib) > l(b)$, since $s_ib$ has one more inversion than $b$, namely $(i,i+1) \mapsto (i+1,i)$.  This contradicts our assumption that $l(s_ib) < l(b)$, thus $b$ must fix some value outside of $\{i,i+1\}$.  Then $s_ib$ necessarily fixes the same value.
\end{enumerate}

Let us consider the two possible cases.  Take first the case when $Q_1$ splits while $Q_2$ does not.  Then by the results of Subsection \ref{ssec:kgb_param_so_odd}, in the weak order graph for $O(2n,\C)$-orbits, any edge joining $Q_1$ to $Q_2$ must be blue.  Indeed, if the involution corresponding to $Q_1$ is fixed point-free, then $s_i b s_i$ is also fixed point-free.  Since $Q_2$ does not split, it corresponds to an involution with fixed points, which obviously cannot be $s_i b s_i$.  The only conclusion is that $s_i b s_i = b$, and that the involution corresponding to $Q_2$ is $s_i b$.  This implies that any edge joining $Q_1$ to $Q_2$ is blue.

The situation in this case turns out to be what one would likely expect:  $Q_1$ splits as components $Q_1'$ and $Q_1''$, and we have
\begin{enumerate}
	\item $Q_1' <_i Q_2$, and the edge is black.
	\item $Q_1'' <_i Q_2$, and the edge is black.
\end{enumerate}

The geometry here is simple:  The restriction of the map $\pi_{\ga_i}: G/B \rightarrow G/P_{\ga_i}$ to $\overline{Q_1}$ is generically $2$-to-$1$.  Over a generic point $gP_{\ga_i}$ in the image, one of the two preimage points will lie in $Q_1'$, and the other will lie in $Q_1''$.  Thus the further restriction of $\pi_{\ga_i}$ to either component of $\overline{Q_1}$ is birational.

Now consider the second case, where both $Q_1$ and $Q_2$ split (say as $Q_1'$, $Q_1''$ and $Q_2'$, $Q_2''$).  In this case, we can see combinatorially that any edge joining $Q_1$ to $Q_2$ must be black.  Indeed, $Q_1$ corresponds to a fixed point-free involution $b$, while $Q_2$ corresponds to a fixed point-free involution $c$ for some $s_i$.  If $s_ibs_i = b$, then $s_i b$ must have fixed points.  Since $c$ is assumed not to have fixed points, we must have that $s_ibs_i = c$.  Thus any edge joining $Q_1$ to $Q_2$ is black.

It follows from \cite[Proposition 7.9, Part (i)]{Richardson-Springer-90} that we should have one of the following two cases:
\begin{enumerate}
	\item $Q_1' <_i Q_2'$ and $Q_1'' <_i Q_2''$ (both edges black)
	\item $Q_1' <_i Q_2''$ and $Q_1'' <_i Q_2'$ (both edges black).
\end{enumerate}

However, it is not obvious (at least to the author) how to tell which is the case once we have fixed our choices of $Q_1'$, $Q_1''$, $Q_2'$, and $Q_2''$.  As a simple example, consider the case $n=2$, with $Q_1$ the bottom orbit corresponding to $w_0$, and $Q_2$ the orbit corresponding to $(1,3)(2,4)$.  As noted above, we have $Q_1 <_1 Q_2$.  It is also the case that $s_3 w_0 s_3 = (1,3)(2,4)$, so $Q_1 <_3 Q_2$ as well.  If we declare, say, that $Q_1'$, $Q_1''$, $Q_2'$, and $Q_2''$ are represented by $\left\langle e_1,e_2,e_3,e_4 \right\rangle$, $\left\langle e_1,e_3,e_2,e_4 \right\rangle$, $\left\langle e_1,e_2,e_4,e_3 \right\rangle$, and $\left\langle e_1,e_3,e_4,e_2 \right\rangle$, respectively, how does one know which of the following four sets of closure relations is correct?
\begin{enumerate}
	\item $Q_1' <_1 Q_2'$, $Q_1' <_3 Q_2'$, $Q_1'' <_1 Q_2''$, $Q_1'' <_3 Q_2''$
	\item  $Q_1' <_1 Q_2'$, $Q_1' <_3 Q_2''$, $Q_1'' <_1 Q_2''$, $Q_1'' <_3 Q_2'$
	\item $Q_1' <_1 Q_2''$, $Q_1' <_3 Q_2''$, $Q_1'' <_1 Q_2'$, $Q_1'' <_3 Q_2'$
	\item $Q_1' <_1 Q_2''$, $Q_1' <_3 Q_2'$, $Q_1'' <_1 Q_2'$, $Q_1'' <_3 Q_2''$
\end{enumerate}

Ultimately, we can answer this question by examining the formulas for the equivariant fundamental classes of these orbit closures and computing their restrictions at $S$-fixed points contained in one orbit closure or another.  In the example given above, we know that the orbit $Q_1'$ is represented by the polynomial $2(y_1y_2 + x_1x_2)(x_1+x_2)(x_1+x_3)$.  Applying $\partial_1$ to this polynomial, we get $2(y_1y_2+x_1x_2)(x_1+x_2)$.  This polynomial must represent either $[Q_2']$ or $[Q_2'']$.  As chosen above, $Q_2'$ is represented by the $S$-fixed point corresponding to $1243$, while $Q_2''$ is represented by the $S$-fixed point corresponding to $1342$.   Computing the restriction of the class $\partial_1([Q_1'])$ at the fixed point $1243$, we get
\[ 2(Y_1Y_2+Y_1Y_2)(Y_1+Y_2) = 4Y_1Y_2(Y_1+Y_2). \]
On the other hand, when we compute the restriction of the class $\partial_1([Q_1'])$ at the fixed point $1342$, we get
\[ 2(Y_1Y_2 + Y_1Y_3)(Y_1+Y_3) = 2(Y_1Y_2 - Y_1Y_2)(Y_1-Y_2) = 0. \]
This tells us that we must have $Q_1' <_1 Q_2'$ (and hence also $Q_1'' <_1 Q_2''$).  Indeed, the computation shows that the $S$-fixed point $1243$ must be contained in the closure of the orbit $s_1 \cdot Q_1'$, or else the restriction of $[s_1 \cdot Q_1']$ at $1243$ would necessarily be zero.  This says $s_1 \cdot Q_1' = Q_2'$.  A similar computation involving $\partial_3([Q_1'])$ shows also that $Q_1' <_3 Q_2'$ and $Q_1'' <_3 Q_2''$.  Thus option (1) above is the correct one.

\subsubsection{Example}
We give the results of the remainder of the computation for the case $n=2$, some of which was worked out in the previous subsection to enhance the clarity of the exposition there.  (We treat both the cases $G=GL(4,\C),K=O(4,\C)$ and $G=SL(4,\C),K=SO(4,\C)$.)  There are 10 involutions in $W$:
\[ id; (1,2); (1,3); (1,4); (2,3); (2,4); (3,4); (1,2)(3,4); (1,3)(2,4); (1,4)(2,3). \]

The weak order graph for $O(4,\C)$-orbits on $X$ is given in Figure \ref{fig:type-a-orthogonal-3} of the appendix, with formulas shown in Table \ref{tab:type-a-o4}.  The only comment we offer on that computation is simply to point out that the formula for the bottom orbit corresponding to $w_0$ is obtained by adding the formulas for the classes of the two irreducible components, those being the two closed $SO(4,\C)$-orbits.

The weak order graph for $SO(4,\C)$-orbits on $X$ is given in Figure \ref{fig:type-a-orthogonal-4}, with formulas shown in Table \ref{tab:type-a-so4}.  All the ideas required for the computation are discussed in the previous subsection, so we offer no further comment here.

\subsection{$K \cong Sp(2n,\C)$}
The final $K$ to consider in type $A$ is $K=Sp(2n,\C)$, which corresponds to the real form $G_\R = SL(n,\QU)$ of $SL(2n,\C)$.  ($\QU$ denotes the quaternions.)  We realize $K$ as the isometry group of the skew form given by $J_{n,n}$ (cf. Subsection \ref{ssec:notation}) --- that is, $K$ is the fixed point subgroup of the involution
\[ \theta(g) = J_{n,n} (g^{-1})^t J_{n,n}. \]

As was the case with the orthogonal groups, one checks easily that given this realization of $K$, the diagonal elements $S = K \cap T$ are a maximal torus of $K$, and the lower-triangular elements $B' = B \cap K$ are a Borel subgroup of $K$.  Also as with the orthogonal groups, we have $\text{rank}(K) < \text{rank}(G)$, so we have a proper inclusion of tori $S \subsetneq T$, and we work over $S$-equivariant cohomology $H_S^*(X)$.  If $X_1,\hdots,X_{2n} \in \frt^*$ are coordinates on $\frt$, restriction to $\frs$ is given by $\rho(X_i) = Y_i$, $\rho(X_{2n+1-i}) = -Y_i$ for $i=1,\hdots,n$.

The roots of $K$ are the following:
\[ \Phi_K = \{\pm (Y_i \pm Y_j) \mid 1 \leq i < j \leq n \} \cup \{\pm 2Y_i \mid i=1,\hdots,n\}. \]

The Weyl group $W_K$ acts on $\frs^*$ as signed permutations of the coordinate functions $\{Y_1,\hdots,Y_n\}$ with any number of sign changes.  $W_K$ embeds into $W$ just as in our prior examples.

\subsubsection{A Formula for the Closed Orbit}
As was the case with $K = SO(2n+1,\C)$, here there is only one closed orbit --- namely, $Q=K \cdot 1B$, the orbit of the $S$-fixed point corresponding to the identity of $W$.  The $S$-fixed points contained in $Q$ correspond to the images of elements of $W_K$ in $W$.

$Q$ being the only closed $K$-orbit, we give a formula for its $S$-equivariant class.  The proof is virtually identical to that given in the case of the odd orthogonal group, except simpler, so we omit it.

\begin{prop}\label{prop:formula_for_closed_sp_orbit}
Let $Q$ be the closed $K$-orbit of the previous proposition.  Then $[Q]$ is represented by
\[ P(x,y) := \displaystyle\prod_{1 \leq i < j \leq n}(x_i + x_j)(x_i + x_{2n+1-j}). \]
\end{prop}

\subsubsection{Parametrization of the Orbits and the Weak Order}
We refer the reader to \cite[Example 10.4]{Richardson-Springer-90}, and to \cite[\S 2.4.2]{Wyser-Thesis} for even further detail.  Here, the $K$-orbits are not in one-to-one correspondence with twisted involutions of $W$, but do inject into them.  The set of twisted involutions of $W$ is once again in bijection with the honest involutions (the bijection again being left-multiplication by $w_0$) so that the $K$-orbits are parametrized by some subset of the involutions of $W$.  The appropriate subset of $W$ turns out to be the set of \textit{fixed point-free} involutions.

When one parametrizes the $K$-orbits by fixed point-free involutions, the unique closed orbit once again corresponds to $w_0$.  From this starting point, the weak order poset can be generated by the following rules:  Given a fixed point-free involution $b$ and a simple reflection $s_i$,

\begin{enumerate}
	\item If $l(s_ib) > l(b)$, or if $s_i b s_i = b$, then $s_i \cdot Q_b = Q_b$.
	\item Else, $Q_b <_i Q_{s_ibs_i}$, and the edge is black.
\end{enumerate}

The parametrization of $K \backslash G/B$ by fixed point-free involutions encodes precisely the same linear algebraic descriptions of the orbits and orbit closures in this case as it does in the case of the orthogonal groups.  Namely, letting $\gamma$ denote the symplectic form with isometry group $K$, if we define $Q_b$ to be
\[ \{ F_{\bullet} \in X \ \vert \ \text{rank}(\gamma|_{F_i \times F_j}) = r_b(i,j) \text{ for all } i,j \},\]
then $Q_b$ is a single $K$-orbit on $G/B$, and the association $b \mapsto Q_b$ defines a bijection between the set of fixed point-free involutions and $K \backslash G/B$.  Moreover,
\[ \overline{Q_b} = \{ F_{\bullet} \in X \mid \text{rank}(\gamma|_{F_i \times F_j}) \leq r_b(i,j) \text{ for all } i,j \}. \]

Indeed, the proof is exactly same as that given for Proposition \ref{prop:ortho-orbit-closures}.

\subsubsection{Example}
We give the details of the computation in the very small case $n = 2$ (so $(G,K)=(SL(4,\C),Sp(4,\C))$).  Here, there are 3 fixed point-free involutions, and hence 3 orbits.  The involutions are $(1,2)(3,4)$, $(1,3)(2,4)$, and $(1,4)(2,3)$.

We start at $w_0=(1,4)(2,3)$ and work upward, applying the rule of the previous subsection:
\[ s_1 w_0 s_1 = (1,3)(2,4), \]
\[ s_2 w_0 s_2 = w_0, \]
\[ s_3 w_0 s_3 = (1,3)(2,4), \]

so $w_0 <_1 (1,3)(2,4)$ and $w_0 <_3 (1,3)(2,4)$.  Next, we move up to $(1,3)(2,4)$, noting that we only need to compute the action of $s_2$:
\[ s_2 (1,3)(2,4) s_2 = (1,2)(3,4), \]

and we are done.  The weak order graph appears as Figure \ref{fig:type-a-symplectic-1} of the appendix.

By Proposition \ref{prop:formula_for_closed_sp_orbit}, the formula for $[Y_{w_0}]$ is $(x_1+x_2)(x_1+x_3)$.  We obtain $[Y_{(1,3)(2,4)}]$ by applying either $\partial_1$ or $\partial_3$.  In either case, the result is $[Y_{(1,3)(2,4)}] = x_1+x_2$.  Finally, we obtain $[Y_{(1,2)(3,4)}]$ by applying $\partial_2$ to $[Y_{(1,3)(2,4)}]$, and of course the result is $[Y_{(1,2)(3,4)}] = 1$.  These formulas appear in Table \ref{tab:type-a-sp4}.

The weak order graph and formulas for the larger example $n=3$ appear in Figure \ref{fig:type-a-symplectic-2} and Table \ref{tab:type-a-sp6}, respectively.  (In that case, there are 15 orbits.)

\section{$K$-orbit Closures as Universal Degeneracy Loci}\label{sec:deg-loci}
In this section, we describe our main application of the formulas obtained in the previous section, realizing the $K$-orbit closures as universal degeneracy loci of a certain type determined by $K$.  We describe a translation between our formulas for equivariant fundamental classes of $K$-orbit closures and Chern class formulas for the fundamental classes of such degeneracy loci.

\subsection{Generalities}
Before handling the specifics of the cases at hand, we first discuss the general setup.  Denote by $E$ a contractible space with a free action of $G$.  Then $E$ also has a free action of $B$, and of $K$, by restriction of the $G$-action.  We shall use the same space $E=EG=EB=EK$ as the total space of a universal principal $G$, $B$, or $K$-bundle, as appropriate.  Denote by $BG$, $BB$, and $BK$ the quotients of $E$ by the actions of $G$, $B$, and $K$, respectively.  These are classifying spaces for the respective groups.

The reason we have worked in $S$-equivariant cohomology $H_S^*(G/B)$ throughout is to take advantage of the localization theorem.  However, the equivariant fundamental classes of $K$-orbit closures in fact live in $K$-equivariant cohomology $H_K^*(G/B)$.  (In the event that $K$ is disconnected, this should be interpreted as $H_{K^0}^*(G/B)$, where $K^0$ denotes the identity component of $K$.)  Indeed, for a $K$-orbit closure $Y$, the $S$-equivariant class $[Y]_S$ is simply the image $\pi^*([Y]_K)$ under the pullback by the natural map
\[ \pi:  E \times^S (G/B) \rightarrow E \times^K (G/B). \]
It is a basic fact about equivariant cohomology that this pullback is injective, and embeds $H_K^*(G/B)$ in $H_S^*(G/B)$ as the $W_K$-invariants (\cite{Brion-98_i}).  Thus $H_K^*(G/B)$ is a subring of $H_S^*(G/B)$, and the $S$-equivariant fundamental classes of $K$-orbit closures live in this subring.  

Now, $H_K^*(G/B)$ is, by definition, the cohomology of the space $E \times^K (G/B)$, and this space is easily seen to be isomorphic to the fiber product $BK \times_{BG} BB$.  (The argument is identical to that given in the proof of Proposition \ref{prop:eqvt-cohom-flag-var} to show that $E \times^S (G/B) \cong BS \times_{BG} BB$ --- simply replace $S$ by $K$.)

Now, suppose that $X$ is a smooth variety, and that $V \rightarrow X$ is a complex vector bundle of rank $n$.  In any event, we have a classifying map $X \stackrel{\rho}{\longrightarrow} BG$ such that $V$ is the pullback $\rho^*(\mathcal{V})$, where $\mathcal{V} = E \times^G \C^n$ is a universal vector bundle over $BG$, with $\C^n$ carrying the natural representation of $G$.

For any closed subgroup $H$ of $G$, $BH \rightarrow BG$ is a fiber bundle with fiber isomorphic to $G/H$.  A lift of the classifying map $\rho$ to $BH$ corresponds to a reduction of structure group to $H$ of the bundle $V$.  Such a reduction of structure group can often be seen to amount to some additional structure on $V$.  For instance, in type $A$, reduction of the structure group of $V$ from $GL(n,\C)$ to the Borel subgroup $B$ of upper-triangular matrices is well-known to be equivalent to $V$ being equipped with a complete flag of subbundles.

We will be concerned with certain structures on $V$ which amount to a reduction of structure group to $K$.  Such a reduction gives us a lift of the classifying map $\rho$ to $BK$.  Suppose that we know what this structure is, and that $V$ possesses this structure, along with a single flag of subbundles $E_{\bullet}$.  Then we have two separate lifts of $\rho$, one to $BK$, and one to $BB$.  Taken together, these two lifts give us a map
\[ X \stackrel{\phi}{\longrightarrow} BK \times_{BG} BB. \]

Our general thought is to consider a subvariety $D$ of $X$ which is defined as a set by linear algebraic conditions imposed on fibers over points in $X$.  These linear algebraic conditions describe the ``relative position" of a flag of subbundles of $V$ and the additional structure on $V$ amounting to the lift of the classifying map to $BK$.  The varieties we consider are precisely those which are set-theoretic inverse images under $\phi$ of (isomorphic images of) $K$-orbit closures in $BK \times_{BG} BB \cong E \times^K (G/B)$.  The linear algebraic descriptions of such a subvariety $D$ come directly from similar linear algebraic descriptions of a corresponding $K$-orbit closure $Y$.  We also realize various bundles on $X$ as pullbacks by $\phi$ of certain tautological bundles on the universal space, so that the Chern classes of the various bundles on $X$ are pullbacks of $S$-equivariant classes represented by the variables $x_i$ and $y_i$, or perhaps polynomials in these classes.

Assuming that our setup is ``suitably generic", by which we mean precisely that
\begin{equation}\label{eqn:pullback}
	[D] = [\phi^{-1}(Y)] = \phi^*([Y]),
\end{equation}
our equivariant formula for $[Y]$ gives us, in the end, a formula for $[D]$ in terms of the Chern classes of the bundles involved.  See \cite[\S B.3, Lemma 5]{Fulton-YoungTableaux} for a sufficient condition to guarantee this for any map $\phi$ of nonsingular varieties.  The genericity requirement should be thought of philosophically as an insistence that the structures on the bundle $V$ be in general position with respect to one another.

With the general picture painted, we now proceed to our specific examples.

\subsection{Examples}\label{ssec:type-a-other-deg-loci}
Here we make explicit the general setup described in the previous subsection in the examples covered in this article.  We start with the case $G = GL(n,\C)$ and $K = O(n,\C)$, with $n$ either even or odd.

The space $BK$ is a $G/K$-bundle over $BG$, with $G/K$ the space of all nondegenerate, symmetric bilinear forms on $\C^n$.  This correspondence associates to the coset $gK \in G/K$ the form $g \cdot \gamma$, with
\[ g \cdot \gamma(v,w) = \gamma(g^{-1}v,g^{-1}w). \]
The form $\gamma$ is the one associated to the coset $1K$, and is defined by
\[ \gamma(e_i,e_j) = \delta_{i,n+1-j} \]
where $e_1,\hdots,e_n$ is the standard basis for $\C^n$.  Then a point $eK \in BK$ can naturally be identified with a quadratic form on the fiber $\caV_{eG}$ in the following way:  Let $v_1,\hdots,v_n = [e,e_1],\hdots,[e,e_n] \in E \times^G \C^n$ be a basis for $\caV_{eG}$, and define the form associated to $eK$ by
\[ \left\langle v_i, v_j \right\rangle = \delta_{i,n+1-j}. \]

It is a standard fact that a vector bundle $V \rightarrow X$ of rank $n$ admits a reduction of structure group to $O(n,\C)$ if and only if the bundle carries a nondegenerate quadratic form.  By this we mean a bundle map $\text{Sym}^2(V) \rightarrow X \times \C$ which restricts to a nondegenerate quadratic form on every fiber.  (We will always assume our forms take values in the trivial line bundle.)  If $\rho: X \rightarrow BG$ is a classifying map for the bundle $V$, then the lift of $\rho$ to $BK$ sends $x \in X$ to the point of $BK$ which represents the form $\gamma|_{V_x} = \gamma|_{\caV_{\rho(x)}}$ on the fiber $\caV_{\rho(x)}$.  Then $\gamma$ is effectively pulled back from a corresponding ``tautological" form $\tau$ on $\pi^* \caV \rightarrow BK$ ($\pi$ the projection $BK \rightarrow BG$), whose values on the fiber of $\pi^* \caV$ over every point of $BK$ are identified by the point itself.

Similarly, $BB$ is a $G/B$-bundle over $BG$.  A point $eB \in BB$ can be naturally identified with a complete flag on the fiber $\caV_{eG}$:  Letting $v_1,\hdots,v_n = [e,e_1],\hdots,[e,e_n] \in E \times^G \C^n$ be a basis for $\caV_{eG}$, the flag associated to $eB \in BB$ is the one whose $i$th subspace is the span of $v_1,\hdots,v_i$.  $BB$ carries a tautological flag of bundles $\caT_{\bullet}$, with $(\caT_i)_{eB}$ being equal to the $i$th subspace of the flag on $\caV_{eG}$ represented by the point $eB$.  A lift of the classifying map $\rho$ to $BB$ amounts to a complete flag $E_{\bullet}$ of the vector bundle $V$, with the flag pulled back from $\caT_{\bullet}$.  Indeed, over any $x \in X$, the fiber $(E_{\bullet})_x$ is precisely $(\caT_{\bullet})_{\rho(x)}$.

Thus we see that given a vector bundle $V$ (with classifying map $\rho$) equipped with a quadratic form $\gamma$ and a complete flag of subbundles $E_{\bullet}$, we get a map $\phi: X \rightarrow BK \times_{BG} BB$ which sends $x \in X$ to the point $(\tau|_{\rho(x)}, (\caT_{\bullet})_{\rho(x)}) = (\gamma|_{V_x}, (E_{\bullet})_x)$.

Now, recall that the $K$-orbits are parametrized by involutions in the case at hand, and that given an involution $b \in S_n$, we have by Proposition \ref{prop:ortho-orbit-closures} that the closure of the corresponding $K$-orbit $Q_b$ is precisely
\begin{equation}\label{eqn:closure-equations}
	\overline{Q_b} = \{F_{\bullet} \mid \text{rank}(\gamma|_{F_i \times F_j}) \leq r_b(i,j) \text{ for all } i,j \}.
\end{equation}

For the sake of brevity, given a form $\gamma$ on a vector space $V$, together with a flag $F_{\bullet}$ on $V$, we say that $\gamma$ ``has rank at most $b$ on the flag $F_{\bullet}$" if the flag satisfies the conditions of (\ref{eqn:closure-equations}) relative to $\gamma$.

We note that if $Y_b = \overline{Q_b} \subseteq G/B$ is a $K$-orbit closure, then the isomorphism between $E \times^K (G/B)$ and $BK \times_{BG} BB$ carries $E \times^K Y_b$ to the set of all (Form, Flag) pairs where the form has rank at most $b$ on the flag.  Indeed, given $gB \in Y_b$, the point $[e,gB] \in E \times^K Y_b$ is carried to the point $(eK,egB) \in BK \times_{BG} BB$.  This point represents the antidiagonal form on $\caV_{eG}$ relative to the basis $[e,e_1],\hdots,[e,e_n]$, together with the flag $gB$ on $\caV_{eG}$ relative to that same basis.  Then the form has rank at most $b$ on the flag, by choice of $gB$.  On the other hand, any point $(eK,egB) \in BK \times_{BG} BB$ where the antidiagonal form on $\caV_{eG}$ has rank at most $b$ on the flag $gB$ is matched up with the point $[e,gB]$, clearly an element of $E \times^K Y_b$.

Given this, together with our description of the map $\phi$, we see that given a vector bundle $V$ over $X$ with a form and a flag, and an involution $b$, the locus
\begin{equation}\label{eqn:deg-locus}
	D_b = \{x \in X \mid \gamma|_{V_x} \text{ has rank at most $b$ on } (F_{\bullet})_x\}
\end{equation}
is precisely $\phi^{-1}(\widetilde{Y_b})$, with $\wt{Y_b}$ the isomorphic image of $E \times^K Y_b$ in $BK \times_{BG} BB$.  Thus generically, the class of such a locus is given by $[D_b] = \phi^*(\widetilde{Y_b})$.

Now, consider the equivariant classes $x_i$.  $G/B$ has a tautological flag of bundles $T_{\bullet}$.  Each bundle in this flag is $K$-equivariant, so that we get a flag of bundles $(T_{\bullet})_K = E \times^K T_{\bullet}$ on $(G/B)_K := E \times^K G/B$.  This flag pulls back to a tautological flag $(T_{\bullet})_S$ on $(G/B)_S$ whose subquotients $(T_i)_S / (T_{i-1})_S$ are the line bundles $E \times^S (G \times^B \C_{X_i})$.  Recall that the classes $x_i$ are precisely the first Chern classes of the latter line bundles.  The bundles $(T_{\bullet})_K$ on $(G/B)_K$ match up with the bundles $\caT_{\bullet}$ on $BK \times_{BG} BB$ via the isomorphism between the two base spaces, and as we have noted, the latter bundles pull back to the flag $F_{\bullet}$ of bundles on $X$.  The upshot is that 
\[ \phi^*(x_i)  = c_1(F_i/F_{i-1}) \]
for $i=1,\hdots,n$.

All this discussion amounts to the following:  Our formulas for the equivariant classes of the $K$-orbit closure $Y_b$, which we note involve \textit{only the $x$ variables}, (generically) give formulas for $[D_b]$ in the Chern classes $c_1(F_i/F_{i-1})$.

Note that the above analysis applies to the case $G = GL(n,\C)$, $K = O(n,\C)$.  The case $G = SL(n,\C)$, $K = SO(n,\C)$ is identical in the event that $n$ is odd, but a bit different in the case that $n$ is even.  We address this in a moment.  First, we point out that the above analysis applies equally well to the case of $G = SL(2n,\C)$, $K = Sp(2n,\C)$, with only very minor modifications.  The orbit closures in that case are parametrized by \textit{fixed point-free} involutions, and descriptions of their closures are identical to those of (\ref{eqn:closure-equations}) when $\gamma$ is taken to be the \textit{skew} form for which $K$ is the isometry group.  A lift of the classifying map to $BK$ then amounts to a nondegenerate \textit{skew} form on the bundle $V$, by which we mean a bundle map $\bigwedge^2(V) \rightarrow X \times \C$ which restricts to a nondegenerate skew form on each fiber.  Given such a form, along with a flag of subbundles of $V$, one can define a degeneracy locus $D_b \subseteq X$ associated to a fixed point-free involution $b$ just as in (\ref{eqn:deg-locus}) above.  And just as above, our formulas for the equivariant classes of $K$-orbit closures (which again involve only the $x$ variables) pull back to a formula for $[D_b]$ in the Chern classes of the subquotients of the flag.

We now address the case of $(SL(2n,\C), SO(2n,\C))$.  In the even case, each $O(2n,\C)$-orbit on $GL(n,\C)/B$ associated to a fixed point-free involution splits as a union of two $SO(2n,\C)$-orbits, so that each $O(2n,\C)$-orbit \textit{closure} has two irreducible components, each the closure of a distinct $SO(2n,\C)$-orbit.  Thus a formula for the class of an $SO(2n,\C)$-orbit closure associated to a fixed point-free involution $b$ should pull back to a formula for an irreducible component of the locus $D_b$, defined as in (\ref{eqn:deg-locus}).  Note (see, e.g., Table \ref{tab:type-a-so4}) that our formulas for equivariant classes of $SO(2n,\C)$-orbit closures associated to involutions with fixed points involve the $x$-variables only, but the formulas for equivariant classes of orbit closures associated to fixed point-free involutions typically also involve the class $y_1 \hdots y_n$.  We now identify this class as pulling back to an ``Euler class" $e \in H^*(X)$ associated to our bundle with quadratic form.

The Euler class of a rank $2n$ complex vector bundle $V \rightarrow X$ with nondegenerate quadratic form is a class $e \in H^{2n}(X)$ which is uniquely defined up to sign by the following property:  If $W \rightarrow Y$ is any rank $2n$ complex vector bundle with nondegenerate quadratic form, possessing a maximal (rank $n$) isotropic subbundle $E$, and if $\rho: Y \rightarrow X$ is a map for which $W = \rho^*V$, then $\rho^*(e) = \pm c_n(E)$.  In particular, the space $BK$ carries the bundle $\caV$ (omitting the pullback notation), equipped with a ``tautological" nondegenerate quadratic form, as we have already noted, so there is an associated Euler class in $H^{2n}(BK)$.  (For the interested reader, we mention that this class is the Euler class --- in the sense of \cite[\S 9]{Milnor-Stasheff} --- of a rank $2n$ \textit{real} bundle on $BK$ whose complexification is $\caV$.  The real bundle in question is pulled back, through a homotopy equivalence $BSO(2n,\C) \rightarrow BSO(2n,\R)$, from the canonical rank $2n$ real bundle $\caV_{\R}$ on the latter classifying space.)  The Euler class of $V \rightarrow X$ is the pullback of this class in $H^{2n}(BK)$ through the classifying map.  Note that it exists even in cases where $V$ does not carry a maximal isotropic subbundle.  This class is \textit{not} a polynomial in the Chern classes of $V$.  (This could indicate that the equivariant classes of $SO(2n,\C)$-orbit closures on $G/B$ associated to fixed point-free involutions are \textit{not} expressible in the $x$-variables alone.)  These facts are explained further in \cite{Edidin-Graham-95} where, among other results, the existence of an \textit{algebraic} Euler class of a Zariski-locally trivial bundle with quadratic form is established.

Now, note that the class $y_1 \hdots y_n \in H_S^*(G/B)$ is (the pullback to $H_S^*(G/B)$ of) $c_n(\bigoplus_{i=1}^n \caL_{Y_i})$ in the notation of Subsection \ref{ssec:eqvt_cohom}, Proposition \ref{prop:eqvt-cohom-flag-var} (again omitting pullback notation).  The bundle $\bigoplus_{i=1}^n \caL_{Y_i}$ is a maximal isotropic subbundle of the pullback of $\caV$ to $BS$ through the projection $BS \rightarrow BK$.  Thus $y_1 \hdots y_n$, viewed as a class in $H_K^*(G/B)$, is an Euler class for $\caV$.  Pulling all the way back to $X$ through the classifying map, we see that $\phi^*(y_1 \hdots y_n)$ is an Euler class for the bundle $V \rightarrow X$.

Summarizing, our formulas for the equivariant classes of $SO(2n,\C)$-orbit closures can be interpreted as formulas for the fundamental classes of irreducible components of degeneracy loci $D_b$ ($b$ a fixed point-free involution) defined as above, expressed in the first Chern classes of the subquotients of the flag of subbundles, together with an Euler class for the bundle with quadratic form.

\appendix
\section*{Appendix:  Weak Order Graphs and Tables of Formulas in Examples}

\begin{figure}[h!]
	\caption{$(SL(3,\C),SO(3,\C))$}\label{fig:type-a-orthogonal-1}
	\centering
	\includegraphics[scale=0.5]{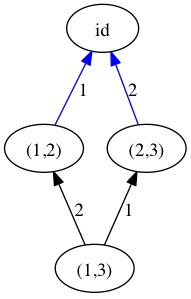}
\end{figure}

\begin{figure}[h!]
	\caption{$(SL(5,\C),SO(5,\C))$}\label{fig:type-a-orthogonal-2}
	\centering
	\includegraphics[scale=0.6]{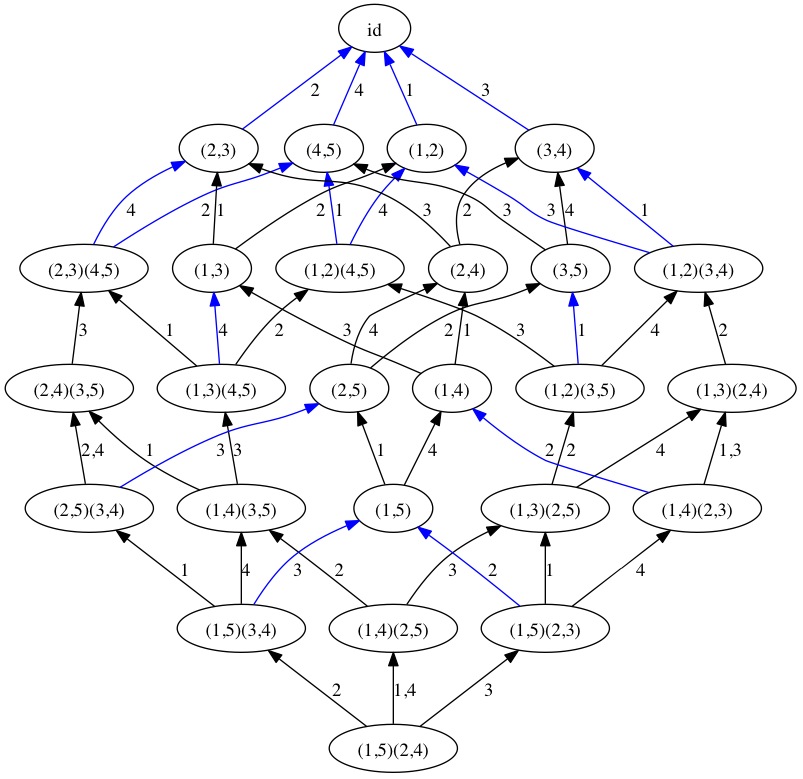}
\end{figure}

\begin{figure}[h!]
	\caption{$(GL(4,\C),O(4,\C))$}\label{fig:type-a-orthogonal-3}
	\centering
	\includegraphics[scale=0.6]{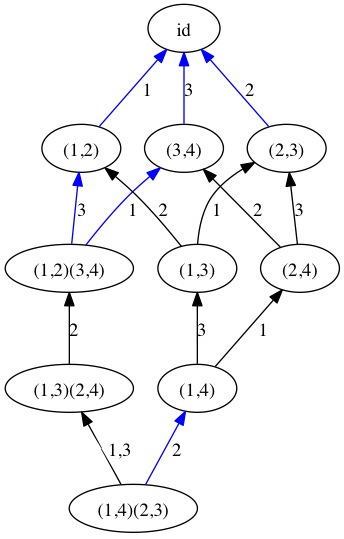}
\end{figure}

\begin{figure}[h!]
	\caption{$(SL(4,\C),SO(4,\C))$}\label{fig:type-a-orthogonal-4}
	\centering
	\includegraphics[scale=0.6]{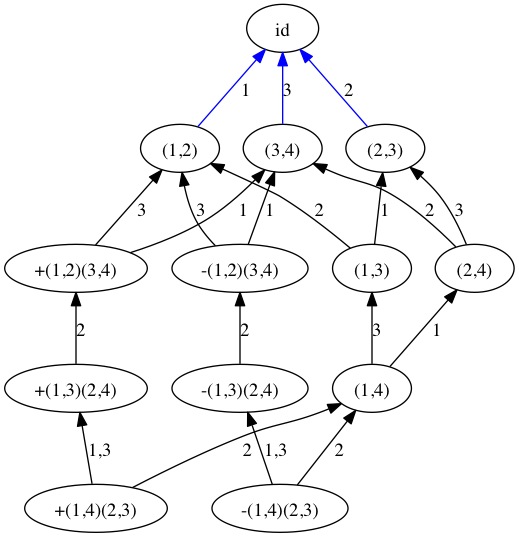}
\end{figure}

\begin{figure}[h!]
	\caption{$(SL(4,\C),Sp(4,\C))$}\label{fig:type-a-symplectic-1}
	\centering
	\includegraphics[scale=0.6]{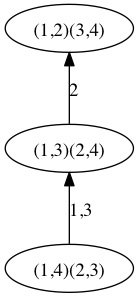}
\end{figure}

\begin{figure}[h!]
	\caption{$(SL(6,\C),Sp(6,\C))$}\label{fig:type-a-symplectic-2}
	\centering
	\includegraphics[scale=0.5]{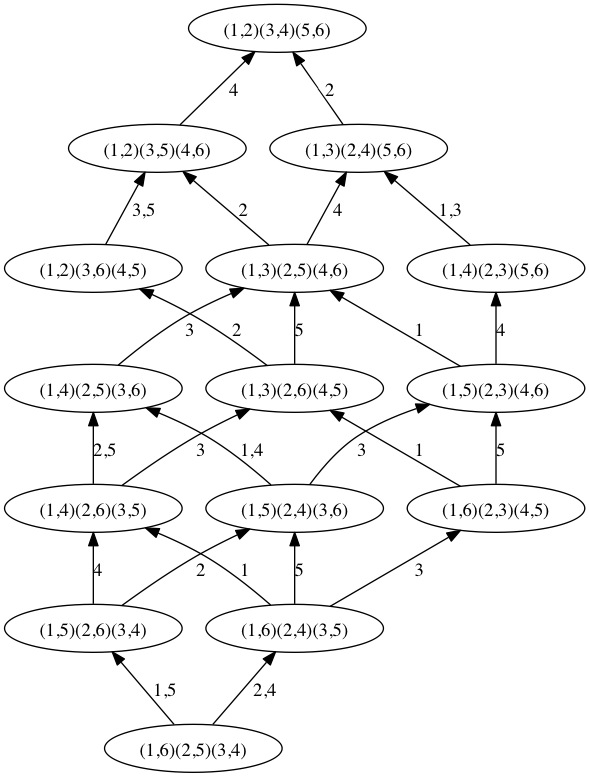}
\end{figure}

\begin{table}[h]
	\caption{Formulas for $(SL(3,\C),SO(3,\C))$}\label{tab:type-a-so3}
	\begin{tabular}{|l|l|}
		\hline
		Involution $\pi$ & Formula for $[Y_{\pi}]$ \\ \hline
		$(1,3)$ & $-2(x_1+x_2)(x_2+x_3)$ \\ \hline
		$(1,2)$ &  $-2(x_2+x_3)$ \\ \hline
		$(2,3)$ & $2(x_1+x_2)$ \\ \hline
		id & $1$ \\ 
		\hline
	\end{tabular}
\end{table}

\begin{table}[h]
	\caption{Formulas for $(SL(5,\C),SO(5,\C))$}\label{tab:type-a-so5}
	\resizebox{18cm}{7cm}{
		\begin{tabular}{|l|l|}
			\hline
			Involution $\pi$ & Formula for $[Y_{\pi}]$ \\ \hline
			$(1,5)(2,4)$ & $4(x_1+x_3)(x_3+x_5)(x_2+x_3)(x_3+x_4)(x_1+x_2)(x_1+x_4)$ \\ \hline
			$(1,5)(3,4)$ & $-4(x_1+x_2)(x_1+x_3)(x_1+x_4)(x_2+x_3)(x_2+x_3+x_4+x_5)$ \\ \hline
			$(1,4)(2,5)$ & $4(x_1+x_2)(x_1+x_3)(x_2+x_3)(x_3+x_4)(x_3+x_5)$ \\ \hline
			$(1,5)(2,3)$ & $4(x_1+x_2)(x_1+x_3)(x_1+x_4)(x_3+x_4)(x_2+x_3+x_4+x_5)$ \\ \hline
			$(2,5)(3,4)$ & $-4(x_1+x_2)(x_1+x_3)(x_2+x_3)(x_3+x_5)$ \\ \hline
			$(1,4)(3,5)$ & $-4(x_1+x_2)(x_1+x_3)(x_2+x_3)(x_2+x_3+x_4+x_5)$ \\ \hline
			$(1,5)$ & $-2(x_1+x_2)(x_1+x_3)(x_1+x_4)(x_2+x_3+x_4+x_5)$ \\ \hline
			$(1,3)(2,5)$ & $4(x_1+x_2)(x_3+x_4)(x_3^2+x_4^2 + \displaystyle\sum_{1 \leq i < j \leq 5} x_ix_j)$ \\ \hline
			$(1,4)(2,3)$ & $4(x_1+x_2)(x_1+x_3)(x_1+x_3+x_4+x_5)(x_2+x_3+x_4+x_5)$ \\ \hline
			$(2,4)(3,5)$ & $4(x_1+x_2)(x_1+x_3)(x_2+x_3)$ \\ \hline
			$(1,3)(4,5)$ & $-4(x_1+x_2)(x_1+x_2+x_3+x_4)(x_2+x_3+x_4+x_5)$ \\ \hline
			$(2,5)$ & $-2(x_1+x_2)(x_1x_2+x_1x_3+x_1x_4+x_1x_5+x_2x_3+x_2x_4+x_2x_5+x_3^2+x_3x_4+x_3x_5+x_4^2+x_4x_5)$ \\ \hline
			$(1,4)$ & $-2(x_1+x_2)(x_1+x_3)(x_2+x_3+x_4+x_5)$ \\ \hline
			$(1,2)(3,5)$ & $-4(x_2^2x_3+x_2x_3^2-x_2x_4^2-x_2x_4x_5-x_3x_4^2-x_3x_4x_5-x_4^3-x_4^2x_5 + (x_1^2+x_1x_2) \displaystyle\sum_{i=2}^5 x_i + x_1x_3 \displaystyle\sum_{i=3}^5 x_i)$ \\ \hline
			$(1,3)(2,4)$ & $4(x_1+x_2)(x_1+x_3+x_4+x_5)(x_2+x_3+x_4+x_5)$ \\ \hline
			$(2,3)(4,5)$ & $4(x_1+x_2)(x_1+x_2+x_3+x_4)$ \\ \hline
			$(1,3)$ & $-2(x_1+x_2)(x_2+x_3+x_4+x_5)$ 	\\ \hline
			$(1,2)(4,5)$ & $-4(x_1+x_2+x_3+x_4)(x_2+x_3+x_4+x_5)$ \\ \hline
			$(2,4)$ & $-2(x_1+x_2)(x_4+x_5)$ \\ \hline
			$(3,5)$ & $-2(x_4+x_5)(x_1+x_2+x_3+x_4)$ \\ \hline
			$(1,2)(3,4)$ & $4(x_4+x_5)(x_2+x_3+x_4+x_5)$ \\ \hline
			$(2,3)$ & $2(x_1+x_2)$ \\ \hline
			$(4,5)$ & $2(x_1+x_2+x_3+x_4)$ \\ \hline
			$(1,2)$ & $-2(x_2+x_3+x_4+x_5)$ \\ \hline
			$(3,4)$ & $-2(x_4+x_5)$ \\ \hline
			id & $1$ \\
			\hline
		\end{tabular}
	}
\end{table}

\begin{table}[h]
	\caption{Formulas for $(GL(4,\C),O(4,\C))$}\label{tab:type-a-o4}
	\begin{tabular}{|l|l|}
		\hline
		Involution $\pi$ & Formula for $[Y_{\pi}]$ \\ \hline
		$(1,4)(2,3)$ & $4x_1x_2(x_1+x_2)(x_1+x_3)$ \\ \hline
		$(1,3)(2,4)$ & $4x_1x_2(x_1+x_2)$ \\ \hline
		$(1,4)$ & $2x_1(x_1+x_2)(x_1+x_3)$ \\ \hline
		$(1,2)(3,4)$ & $4x_1(x_1+x_2+x_3)$ \\ \hline
		$(1,3)$ & $2x_1(x_1+x_2)$ \\ \hline
		$(2,4)$ & $2(x_1+x_2)(x_1+x_2+x_3)$ \\ \hline
		$(1,2)$ & $2x_1$ \\ \hline
		$(3,4)$ & $2(x_1+x_2+x_3)$ \\ \hline
		$(2,3)$ & $2(x_1+x_2)$ \\ \hline
		id & $1$ \\ 
		\hline
	\end{tabular}
\end{table}

\begin{table}[h]
	\caption{Formulas for $(SL(4,\C),SO(4,\C))$}\label{tab:type-a-so4}
	\begin{tabular}{|l|l|l|}
		\hline
		Parameter for $Q$ & Representative for $Q$ & Formula for $[Y]$ \\ \hline
		$+(1,4)(2,3)$ & $\left\langle e_1,e_2,e_3,e_4 \right\rangle$ & $2(x_1x_2+y_1y_2)(x_1+x_2)(x_1+x_3)$ \\ \hline
		$-(1,4)(2,3)$ &  $\left\langle e_1,e_3,e_2,e_4 \right\rangle$ & $2(x_1x_2-y_1y_2)(x_1+x_2)(x_1+x_3)$ \\ \hline
		$+(1,3)(2,4)$ & $\left\langle e_1,e_2,e_4,e_3 \right\rangle$ & $2(x_1x_2+y_1y_2)(x_1+x_2)$ \\ \hline
		$-(1,3)(2,4)$ & $\left\langle e_1,e_3,e_4,e_2 \right\rangle$ & $2(x_1x_2-y_1y_2)(x_1+x_2)$ \\ \hline
		$(1,4)$ & $\left\langle e_1,e_2+e_3,e_2-e_3,e_4 \right\rangle$ & $2x_1(x_1+x_2)(x_1+x_3)$ \\ \hline
		$+(1,2)(3,4)$ & $\left\langle e_1,e_4,e_2,e_3 \right\rangle$ & $2(y_1y_2+x_1^2+x_1x_2+x_1x_3)$ \\ \hline
		$-(1,2)(3,4)$ & $\left\langle e_1,e_4,e_3,e_2 \right\rangle$ & $-2(y_1y_2-x_1^2-x_1x_2-x_1x_3)$ \\ \hline
		$(1,3)$ & $\left\langle e_1,e_2+e_3,e_4,e_2-e_3 \right\rangle$ & $2x_1(x_1+x_2)$ \\ \hline
		$(2,4)$ & $\left\langle e_2+e_3,e_1,e_2-e_3,e_4 \right\rangle$ & $2(x_1+x_2)(x_1+x_2+x_3)$ \\ \hline
		$(1,2)$ & $\left\langle e_1,e_4,e_2+e_3,e_2-e_3 \right\rangle$ & $2x_1$ \\ \hline
		$(3,4)$ & $\left\langle e_2+e_3,e_2-e_3,e_1,e_4 \right\rangle$ & $2(x_1+x_2+x_3)$ \\ \hline
		$(2,3)$ & $\left\langle e_2+e_3,e_1,e_4,e_2-e_3 \right\rangle$ & $2(x_1+x_2)$ \\ \hline
		id & $\left\langle e_1+e_4,e_1-e_4,e_2+e_3,e_2-e_3 \right\rangle$ & $1$ \\
		\hline
	\end{tabular}
\end{table}

\begin{table}[h]
	\caption{Formulas for $(SL(4,\C),Sp(4,\C))$}\label{tab:type-a-sp4}
	\begin{tabular}{|c|l|}
		\hline
		Involution $\pi$ & Formula for $[Y_{\pi}]$ \\ \hline
		$(1,4)(2,3)$ & $(x_1+x_2)(x_1+x_3)$ \\ \hline
		$(1,3)(2,4)$ & $x_1+x_2$ \\ \hline
		$(1,2)(3,4)$ & $1$ \\
		\hline
	\end{tabular}
\end{table}

\begin{table}[h]
	\caption{Formulas for $(SL(6,\C),Sp(6,\C))$}\label{tab:type-a-sp6}
	\begin{tabular}{|c|l|}
		\hline
		Involution $\pi$ & Formula for $[Y_{\pi}]$ \\ \hline
		$(1,6)(2,5)(3,4)$ & $(x_1+x_2)(x_1+x_5)(x_1+x_3)(x_1+x_4)(x_2+x_3)(x_2+x_4)$ \\ \hline
		$(1,5)(2,6)(3,4)$ & $(x_1+x_2)(x_1+x_3)(x_1+x_4)(x_2+x_3)(x_2+x_4)$ \\ \hline
		$(1,6)(2,4)(3,5)$ & $(x_1+x_2)(x_1+x_5)(x_1+x_3)(x_1+x_4)(x_2+x_3)$ \\ \hline
		$(1,4)(2,6)(3,5)$ & $(x_1+x_2)(x_1+x_3)(x_2+x_3)(x_1+x_2+x_4+x_5)$ \\ \hline		
		$(1,5)(2,4)(3,6)$ & $(x_1+x_2)(x_1+x_3)(x_1+x_4)(x_2+x_3)$ \\ \hline
		$(1,6)(2,3)(4,5)$ & $(x_1+x_2)(x_1+x_5)(x_1+x_3)(x_1+x_4)$ \\ \hline
		$(1,4)(2,5)(3,6)$ & $(x_1+x_2)(x_1+x_3)(x_2+x_3)$ \\ \hline
		$(1,3)(2,6)(4,5)$ & $(x_1+x_2)(x_1^2+x_2^2+\displaystyle\sum_{1 \leq i < j \leq 5} x_ix_j)$ \\ \hline
		$(1,5)(2,3)(4,6)$ & $(x_1+x_2)(x_1+x_3)(x_1+x_4)$ \\ \hline		
		$(1,2)(3,6)(4,5)$ & $(x_1+x_2+x_3+x_4)(x_1+x_2+x_3+x_5)$ \\ \hline
		$(1,3)(2,5)(4,6)$ & $(x_1+x_2)(x_1+x_2+x_3+x_4)$ \\ \hline
		$(1,4)(2,3)(5,6)$ & $(x_1+x_2)(x_1+x_3)$ \\ \hline
		$(1,2)(3,5)(4,6)$ & $x_1+x_2+x_3+x_4$ \\ \hline
		$(1,3)(2,4)(5,6)$ & $x_1+x_2$ \\ \hline
		$(1,2)(3,4)(5,6)$ & $1$ \\
		\hline
	\end{tabular}
\end{table}

\bibliographystyle{alpha}
\bibliography{sourceDatabase}

\begin{thebibliography}{{Wys}12}

\bibitem[Bri98]{Brion-98_i}
Michel Brion.
\newblock Equivariant cohomology and equivariant intersection theory.
\newblock In {\em Representation theories and algebraic geometry ({M}ontreal,
  {PQ}, 1997)}, volume 514 of {\em NATO Adv. Sci. Inst. Ser. C Math. Phys.
  Sci.}, pages 1--37. Kluwer Acad. Publ., Dordrecht, 1998.
\newblock Notes by Alvaro Rittatore.

\bibitem[Bri99]{Brion-99}
M.~Brion.
\newblock Rational smoothness and fixed points of torus actions.
\newblock {\em Transform. Groups}, 4(2-3):127--156, 1999.
\newblock Dedicated to the memory of Claude Chevalley.

\bibitem[Bri01]{Brion-01}
Michel Brion.
\newblock On orbit closures of spherical subgroups in flag varieties.
\newblock {\em Comment. Math. Helv.}, 76(2):263--299, 2001.

\bibitem[EG95]{Edidin-Graham-95}
Dan Edidin and William Graham.
\newblock Characteristic classes and quadric bundles.
\newblock {\em Duke Math. J.}, 78(2):277--299, 1995.

\bibitem[Ful92]{Fulton-92}
William Fulton.
\newblock Flags, {S}chubert polynomials, degeneracy loci, and determinantal
  formulas.
\newblock {\em Duke Math. J.}, 65(3):381--420, 1992.

\bibitem[Ful96a]{Fulton-96_2}
William Fulton.
\newblock Determinantal formulas for orthogonal and symplectic degeneracy loci.
\newblock {\em J. Differential Geom.}, 43(2):276--290, 1996.

\bibitem[Ful96b]{Fulton-96_1}
William Fulton.
\newblock Schubert varieties in flag bundles for the classical groups.
\newblock In {\em Proceedings of the {H}irzebruch 65 {C}onference on
  {A}lgebraic {G}eometry ({R}amat {G}an, 1993)}, volume~9 of {\em Israel Math.
  Conf. Proc.}, pages 241--262, Ramat Gan, 1996. Bar-Ilan Univ.

\bibitem[Ful97]{Fulton-YoungTableaux}
William Fulton.
\newblock {\em Young tableaux}, volume~35 of {\em London Mathematical Society
  Student Texts}.
\newblock Cambridge University Press, Cambridge, 1997.
\newblock With applications to representation theory and geometry.

\bibitem[Ful98]{Fulton-IT}
William Fulton.
\newblock {\em Intersection theory}, volume~2 of {\em Ergebnisse der Mathematik
  und ihrer Grenzgebiete. 3. Folge. A Series of Modern Surveys in Mathematics
  [Results in Mathematics and Related Areas. 3rd Series. A Series of Modern
  Surveys in Mathematics]}.
\newblock Springer-Verlag, Berlin, second edition, 1998.

\bibitem[Gra97]{Graham-97}
William Graham.
\newblock The class of the diagonal in flag bundles.
\newblock {\em J. Differential Geom.}, 45(3):471--487, 1997.

\bibitem[Hum75]{Humphreys-75}
James~E. Humphreys.
\newblock {\em Linear algebraic groups}.
\newblock Springer-Verlag, New York, 1975.
\newblock Graduate Texts in Mathematics, No. 21.

\bibitem[Mat79]{Matsuki-79}
Toshihiko Matsuki.
\newblock The orbits of affine symmetric spaces under the action of minimal
  parabolic subgroups.
\newblock {\em J. Math. Soc. Japan}, 31(2):331--357, 1979.

\bibitem[MS74]{Milnor-Stasheff}
John~W. Milnor and James~D. Stasheff.
\newblock {\em Characteristic classes}.
\newblock Princeton University Press, Princeton, N. J., 1974.
\newblock Annals of Mathematics Studies, No. 76.

\bibitem[MT09]{McGovern-Trapa-09}
William~M. McGovern and Peter~E. Trapa.
\newblock Pattern avoidance and smoothness of closures for orbits of a
  symmetric subgroup in the flag variety.
\newblock {\em J. Algebra}, 322(8):2713--2730, 2009.

\bibitem[RS90]{Richardson-Springer-90}
R.~W. Richardson and T.~A. Springer.
\newblock The {B}ruhat order on symmetric varieties.
\newblock {\em Geom. Dedicata}, 35(1-3):389--436, 1990.

\bibitem[RS93]{Richardson-Springer-92}
R.~W. Richardson and T.~A. Springer.
\newblock Combinatorics and geometry of {$K$}-orbits on the flag manifold.
\newblock In {\em Linear algebraic groups and their representations ({L}os
  {A}ngeles, {CA}, 1992)}, volume 153 of {\em Contemp. Math.}, pages 109--142.
  Amer. Math. Soc., Providence, RI, 1993.

\bibitem[RS94]{Richardson-Springer-94}
R.~W. Richardson and T.~A. Springer.
\newblock Complements to: ``{T}he {B}ruhat order on symmetric varieties''
  [{G}eom. {D}edicata {\bf 35} (1990), no. 1-3, 389--436; {MR}1066573
  (92e:20032)].
\newblock {\em Geom. Dedicata}, 49(2):231--238, 1994.

\bibitem[Spr85]{Springer-85}
T.~A. Springer.
\newblock Some results on algebraic groups with involutions.
\newblock In {\em Algebraic groups and related topics ({K}yoto/{N}agoya,
  1983)}, volume~6 of {\em Adv. Stud. Pure Math.}, pages 525--543.
  North-Holland, Amsterdam, 1985.

\bibitem[{Wys}12]{Wyser-Thesis}
Benjamin~J. {Wyser}.
\newblock {\em Symmetric subgroup orbit closures on flag varieties: Their
  equivariant geometry, combinatorics, and connections with degeneracy loci}.
\newblock PhD thesis, University of Georgia, 2012.

\end{thebibliography}

\end{document}